\documentclass[11pt]{article}
\usepackage{amsfonts}
\usepackage{amsmath}
\usepackage{amsthm}
\usepackage[mathscr]{euscript}
\usepackage{mathrsfs}
\usepackage{graphicx}
\usepackage{subfigure}
\usepackage{varwidth}
\usepackage{color}\usepackage{enumitem}
\usepackage{url}
\usepackage{authblk}

\newcommand\kk{\left}
\newcommand\rr{\right}
\newcommand\nm{\nonumber}
\numberwithin{equation}{section}

\newtheorem{proposition}{\textbf{Proposition}}[section]
\newtheorem{definition}{\textbf{Definition}}[section]
\newtheorem{theorem}{\textbf{Theorem}}[section]
\newtheorem{lemma}{\textbf{Lemma}}[section]

\newtheorem{remark}{\textbf{Remark}}[section]

\newcommand\bes{\begin{eqnarray}} \newcommand\ees{\end{eqnarray}}
\newcommand{\bess}{\begin{eqnarray*}}
\newcommand{\eess}{\end{eqnarray*}}

\newcommand\dd{\displaystyle}
\newcommand\ds{{\rm d}s}

\title{Life span of solutions to a semilinear parabolic equation on locally finite graphs\footnote{The first author was supported by NSFC Grant 12201184 and China Postdoctoral Science Foundation Grant 2022M711045, and the second author was supported by NSFC Grant 12171120.}}
\author[a]{Yuanyang Hu}

\author[,\,b]{Mingxin Wang\footnote{Corresponding author. {\sl E-mail}: mxwang@hpu.edu.cn}}

\affil[a]{\small\it School of Mathematics and Statistics, Henan University, Kaifeng, Henan 475004, China}\vspace{2mm}

\affil[b]{\small School of Mathematics and Information Science, Henan Polytechnic University, Jiaozuo 454000, China}

\date{}

\begin{document}
\pagestyle{myheadings}
\maketitle
\vspace{-1cm}

\begin{quote}
\noindent{\bf Abstract.} Let $G=(V,E)$ be a locally finite connected graph. We develop the first eigenvalue method on $G$ introduced in 1963 by Kaplan \cite{Kaplan} on Euclidean space, the discrete Phragm\'{e}n-Lindel\"{o}f principle of parabolic equations and upper and lower solutions method  on $G$.  Using these methods, we establish the estimates and asymptotic behaviour of the life span of solutions to a semilinear heat equation with initial data $\lambda\psi(x)$ for different scales of $\lambda$ on $G$ under some different conditions. Our results are different from the continuous case, which is related to the structure of the graph $G$.

\noindent{\bf Keywords:} Heat equations on graphs; Blow up; Lifespan; Asymptotic behaviour; Maximum principle.

\noindent {\bf AMS subject classifications (2020)}: 35A01, 35K91, 35R02, 58J35.
 \end{quote}

\setlength{\baselineskip}{16pt} \pagestyle{myheadings}
\section{Introduction}
{\setlength\arraycolsep{2pt}

Consider the following Cauchy problem
\begin{equation}\label{1.1}
\begin{cases}
u_t=\Delta_{\mathbb{R}^N}u+u^p\; &\text{in}\;\; \mathbb{R}^N \times(0,T), \\ u(x,0)=\phi(x)\;\; &\text{on}\;\;\mathbb{R}^N,
\end{cases}\end{equation}
where $\Delta_{\mathbb{R}^N}$ is the Laplace operator, $p>1$, $T>0$, $\phi(x)$ is a nonnegative, nontrival, bounded and continuous function on $\mathbb{R}^N$. In the pioneering paper \cite{Fujita}, Fujita proved that the local classical solution of \eqref{1.1} blows up in finite time when $1<p<1+\frac{2}{N}$, while \eqref{1.1} admits a global classical solution when $1+\frac{2}{N}<p$ and $\phi(x)$ is small. Later, the nonexistence of nontrivial global solutions of \eqref{1.1} when $p=1+\frac{2}{N}$ was proved by \cite{Hayakawa, Kobayashi}.

It is well-known that \eqref{1.1} admits a unique nonnegative classical solution $\tilde{u}$, which is local. Define the life span of the solution $\tilde{u}$ as
$$T[\phi]=\sup \left\{T>0 :\,\right. (1.1)\text{ possesses~a~unique~non-negative ~classical~solution in }\left.\mathbb{R}^N \times[0, T)\right\}.$$ Lee-Ni \cite{LeeNi} showed that, if $\liminf_{|x|\rightarrow+\infty}\phi(x)>0$, then
$T[\lambda \phi]<+\infty$ for any $\lambda>0$ and there exist positive constants $C_1$ and $C_2$
so that $C_1 \lambda^{1-p} \leq T[\lambda \phi] \le C_2 \lambda^{1-p}$ for sufficiently small $\lambda>0$.
Furthermore, Gui-Wang \cite{Guiw} proved that $\lim_{\lambda\rightarrow \infty} T[\lambda \phi] \lambda^{p-1}=\frac{1}{p-1}\|\phi\|_{L^\infty(\mathbb{R}^{N})}^{-(p-1)}$, and if
$\lim_{|x|\rightarrow \infty} \phi(x)=\phi_{\infty}>0$, then $\lim_{\lambda \rightarrow 0} T[\lambda \phi] \lambda^{p-1}=\frac{1}{p-1}\phi_{\infty}^{-(p-1)}$.

Let $\mathbb{H}^N$ be the $N$-dimensional hyperbolic space. Bundle-Pozio-Tesei [1] and Wang-Yin \cite{WYn} studied the Cauchy problem
\begin{equation}\label{1.2}
\begin{cases}
 u_t=\Delta_{\mathbb{H}^N} u+h(t) u^p & \text { in } \mathbb{H}^N \times(0, T), \\ u(x,0)=\phi(x) \ge,\not\equiv 0 & \text { on } \mathbb{H}^N,
\end{cases}
\end{equation}
where $\Delta_{\mathbb{H}^N}$ stands for the Laplace-Beltrami operator on $\mathbb{H}^N$, $\phi$ is a  bounded and continuous function on $\mathbb{H}^N$.
They discovered that, in order to produce a "Fujita" phenomenon, it is necessary to take a weight function $h(t)={\rm e}^{\mu t}$ in \eqref{1.2} compared to the problem \eqref{1.1}, where $\mu>0$.
Let $p_H^*=1+\frac{\mu}{\lambda_0}$ with $\lambda_0=\frac{(N-1)^2}{4}$. Then they proved if $1<p<p_H^*$, any nontrivial solution of the problem \eqref{1.2} with $h(t)={\rm e}^{\mu t}$ blows up in finite time, and if $p\ge p_H^*$, the problem \eqref{1.2} admits global solutions for sufficiently small data with $h(t)={\rm e}^{\mu t}$. This is different from what happens in the Euclidean setting addressed by \cite{Fujita, Hayakawa, Kobayashi}.

Later, Wang-Yin \cite{WY} studied the Cauchy problem
\begin{equation}\label{1.3}
\left\{\begin{array}{lll}
u_t=\Delta_{\mathbb{H}^N} u+{\rm e}^{\mu t} u^p\;\; &\text{in}& \mathbb{H}^N \times\left(0, T[\lambda \phi]\right), \\[1mm]
u(x,0)=\lambda\phi(x)  &\text{on}& \mathbb{H}^N,
\end{array}\right.
\end{equation}
where $\lambda$ is a positive parameter, $\phi\not\equiv 0$ is a non-negative bounded and continuous function on $\mathbb{H}^N$, and $T[\lambda \phi]$ is the lifespan of the solution $u$ to the problem \eqref{1.3}. They showed that $\lim _{\lambda \rightarrow \infty} \lambda^{p-1} T[\lambda \phi]=\frac{1}{p-1}\|\phi\|_{L^\infty(\mathbb{H}^{N})}^{-(p-1)}$, if $1<p<p_H^*$, then there exist $C_1, C_2>0$ so that
$$
C_1 \ln \frac{1}{\lambda} \le T[\lambda \phi]\le C_2 \ln \frac{1}{\lambda}\;\; \text { as }\; \lambda \rightarrow 0 \text {, }
$$
and if $p \ge p_H^*$ and  $u_0$ decays more slowly than a natural exponential function at infinity, then $\tilde{C}_1 \ln \frac{1}{\lambda} \leq T_\lambda \leq \tilde{C}_2 \ln \frac{1}{\lambda}$ as $\lambda \rightarrow 0$ for some constants $\tilde{C}_{1}$, $\tilde{C}_{2}>0$. From the conclusions in \cite{WY}, the life span of the solution to the problem on $\mathbb{H}^N$ and the corresponding problem on $\mathbb{R}^N$ (see \cite{Guiw}) have different growth rates as $\lambda \rightarrow 0$.

In recent years, there has been a growing interest in non-existence and existence of global solutions of the Cauchy problem on miscellaneous space, such as graphs (\cite{Len}), metric measure spaces (\cite{FHS}) and manifolds (\cite{RY}).

In order to make our statement more clear, we make some preliminaries firstly. Let $G=(V,E)$ be a graph with the vertex set $V$ and the edge set $E$. For $x,y\in V$, let $xy$ be the edge from $x$ to $y$. We write $y\sim x$ if $xy\in E$. Let $\omega: V \times V\to[0,\infty)$ be an edge weight function satisfying $\omega_{xy}=\omega_{yx}$ for all $x,y \in V$ and $\omega_{xy}>0$ iff $x\sim y$.  For each point $x\in V$, define its degree $${\rm deg}(x)=\#\{y \in V:\, x\sim y\},$$
that is, deg$(x)$ is the number of the edges with endpoint $x$. A graph $G=(V,E)$  is called locally finite if deg$(x)$ is a finite number for each point $x\in V$. Let $\mu: V\to (0,\infty)$ be a positive measure. We also write the graph $G$ as a quadruple $G=G(V,E,\omega,\mu)$. Throughout this paper, unless otherwise stated, we always assume that $G=G(V,E,\omega,\mu)$ is an infinite locally finite connected graph without loops or without multiple edges. Moreover, we require $\omega_{\min}=\inf_{e\in E} \omega_{e}>0$.

\begin{lemma}\label{l1.4}{\rm(\cite[Lemma 1.4]{gri09})} If $G=(V,E)$ is a locally finite connected graph, then the set of vertices $V$ is either finite or countable.
\end{lemma}

Denote the space of real-valued functions on $V$ by $V^{\mathbb{R}}$.
For any $g\in V^{\mathbb{R}}$, define the integral of $g$ on $V$ by
$\int_V g{\rm d}\mu=\sum_{x\in V}g(x)\mu(x)$. Given a weight and a measure, we define
$$
\mu_{\max }=\sup_{x\in V}\mu(x),\;\;\; D_{\omega}=\frac{\mu_{\max}}{\omega_{\min}},
\;\;\;\mbox{and}\;\;D_\mu=\sup_{x\in V}\frac{m(x)}{\mu(x)},$$
where $m(x)=\sum\limits_{y \in V:\,y\sim x}\omega_{xy},\;\,x\in V$.
Define the distance $d(x,y)$ by the smallest number of edges of a path between two vertices $x$ and $y$. We define balls centered at $x$ with radius $r$: $	 B_x^r=\{y\in V:d(x,y)\le r\}$. The volume of a subset $A\subset V$ is defined as
$\mathcal{V}(A)=\sum_{x\in A}\mu(x)$.
We usually write $\mathcal{V}(B_x^r)$ by $\mathcal{V}(x,r)$.

 Consider the following Cauchy problem
\begin{equation}\label{1.4}
\left\{\begin{array}{lll}
u_t=\Delta u+u^p &\text{ in }& V \times(0,+\infty), \\[1mm]
u(x, 0)=\phi(x)\ge,\not\equiv0\;\; &\text{ on }& V,
\end{array}\right.
\end{equation}
where $\phi$ is a bounded function on $V$ and $\Delta$ is the usual graph Laplace operator on G defined by
\begin{equation}\label{1.5}
    \Delta u(x)=\frac1{\mu(x)}\sum_{y\in V}\omega_{x y}(u(y)-u(x)),~x\in V.
\end{equation}
Under the assumption that $G$ satisfies ${CDE}^{\prime}(n, 0)$ (cf. Definition \ref{d2.2}) and uniform polynomial volume growth of degree $m$, Lin-Wu \cite{LWcv} proved that if $1<p<\frac{2}{m}+1$, then any solution of \eqref{1.4} blows up in finite time, and if $1+\frac{2}{m}<p$, then there exists a nonnegative global solution to \eqref{1.4} for a sufficiently small initial data. And most remarkably, the behaviors of the solutions for the problem \eqref{1.4} strongly depend on $m$ and $p$.

In recent years, increasing efforts have been devoted to the development and analysis of partial differential equations on graphs. Authors of \cite{gliny2, Gekazdan,GeJiaKazdan, KS} studied Kazdan-Warner equations on locally finite graphs. For the counterpart of Yamabe type equations, see \cite{gliny1, PS, zl1Pro} and references therein. For the aspect of reaction diffusion equations on graphs, see \cite{BS, CLC, Len, LZ, Tian} and references cited therein. On the discrete time-dependent Schr\"{o}dinger equations, we recommend the readers to \cite{BK, EJ, FB3, FJ, JLMP}. Bauer et al. \cite{BHLLMY} and Horn et al. \cite{HLLY} established Gaussian estimates for the heat kernel on locally finite graphs. For the other study of the heat kernel on graphs, we refer the readers to see \cite{gri09, KL, AW, RW}.

In this paper, we study the Cauchy problem
\begin{equation}\label{1.6}\left\{\!\begin{array}{ll}
u_t=\Delta u+u^{p}\;\; &\text{in}\;\; V\times(0,T) ,\\[1mm]
u(x,0)=\lambda\psi(x)\;\; &\text{on}\;\; V,
\end{array}\right.
\end{equation}
where $G=G(V,E,\omega,\mu)$ is a locally finite connected weighted graph, $\Delta$ is the usual graph Laplacian defined by \eqref{1.5}, $p>1$ and $\lambda>0$ are parameters, $T>0$, and
\begin{equation}\label{1.7}
	\psi: V\to [0,\infty)\; \text{is a bounded function, but not identically zero}.
\end{equation}

Throughout this paper, for an interval $I\subset \mathbb{R}$ and positive integer $n$, we define
  \bess
  C^n_V(I)&=&\{f: V\times I\to \mathbb{R}:\,f(x,\cdot)\in C^n(I)\; \;{\rm for \; each}\;\, x\in V\},\\
  L^1_V(I)&=&\{f: V\times I\to \mathbb{R}:\,f(x,\cdot)\in L^1(I)\; \;{\rm for \; each}\;\, x\in V\}.
  \eess
For a function $u:\,V\times[0,T)\to\mathbb{R}$, when the term $u_t$ appears, we always think that $u\in C^n_V(0,T)$.

\begin{definition} A function $u=u(x,t;\lambda\psi):\,V\times[0,T)\to\mathbb{R}^+$ is called a solution of \eqref{1.6} in $[0,T)$ if $u$ satisfies \eqref{1.6}, $u$ is continuous with respect to $t\in[0,T']$ for any fixed $x\in V$ and $u\in L^\infty(V\times[0,T'])$ for any $0<T'<T$.
\end{definition}

By upper and lower solutions method (see Theorem \ref{t3.2} below) and Lemma \ref{l3.2}, we see that \eqref{1.6} admits a unique local solution $u(x,t;\lambda\psi)$ in $V\times [0,T_0)$ for some $T_0>0$.

We define the life span (maximum existence time) of the solution $u(x,t;\lambda\psi)$ of \eqref{1.6} by
\begin{equation}\label{1.8}
T_\lambda=T_\lambda(G,\Delta)=\text{sup}\{T>0: u(x,t;{\lambda\psi})~\text{solves}~\eqref{1.6}~\text{in}~[0,T)\}.
\end{equation}
In fact, $T_\lambda $ is the blow up time of the solution $u(x,t;\lambda\psi)$ when $T_\lambda <\infty$.

\begin{theorem}	\label{t2.1} If $D_{\mu}<+\infty$, then there exists $\Lambda>0$ such  that when $\lambda>\Lambda$, $T_{\lambda}<\infty$ and
 \begin{equation*}
	\lim\limits_{\lambda\to \infty}\lambda^{p-1}T_\lambda =\frac1{(p-1)\|\psi\|_{\ell^{\infty}{(V)}}^{p-1}}.
	\end{equation*}
\end{theorem}

Interestingly, compared to the continuous cases proposed by \cite{Guiw} and \cite{WY}, there is no geometric conditions regarding the graph $G$ in Theorem \ref{t2.1}.

Let $\mathbb{Z}^N$ be the $N$-dimensional lattice. A homogeneous tree $T_q$ of degree $q+1$ is defined to be a connected graph with no loops, in which every vertex is adjacent to $q+1$ other vertices. The corresponding graph Laplace operators on $\mathbb{Z}^N$ and $T_q$ can be defined by
 \bess
  \Delta_{\mathbb{Z}^N}u(x)=\frac{1}{N}\!\!\sum_{y\in\mathbb{Z}^N:|y-x|=1}\!\![u(y)-u(x)], \;\;\;{\rm and}\;\;
  \Delta_{T_q}u(x)=\!\!\sum_{y\in T_{q}: d(y, x)=1}\!\!\frac{u(y)-u(x)}{q+1},
 \eess
respectively. Consider the problems
\begin{equation}\label{1.9}
    \begin{cases}u_t=\Delta_{\mathbb{Z}^N}u+u^p, & (x, t) \in \mathbb{Z}^N\times(0, T), \\
    u(x, 0)=\lambda \psi(x), & x \in \mathbb{Z}^N,\end{cases}
\end{equation}
and
\begin{equation}\label{1.10}
\begin{cases}
u_t=\Delta_{T_q}u+u^p,~(x,t)\in T_q\times(0,T),\\
        u(x,0)=\lambda \psi(x),~x\in T_q.
  \end{cases}
\end{equation}
The problem \eqref{1.9} is a spatial-discrete version of \eqref{1.1} with $\phi(x)$ replaced with $\lambda\psi(x)$, whereas \eqref{1.10} is a spatial-discrete version of \eqref{1.2} with $h(t)\equiv 1$ with $\phi(x)$ replaced with $\lambda\psi(x)$. Applying Theorem \ref{t2.1} to \eqref{1.9} and \eqref{1.10}, we get
 \bess\lim_{\lambda\to \infty}\lambda^{p-1}T_{\lambda}(\mathbb{Z}^N) =\frac1{(p-1)\|\psi\|_{\ell^{\infty}{(\mathbb{Z}^N)}}^{p-1}},\;\;\;
\lim_{\lambda\to\infty}\lambda^{p-1}T_{\lambda}(T_q)
 =\frac1{(p-1)\|\psi\|_{\ell^{\infty}{(T_q)}}^{p-1}},\eess
where $T_{\lambda}$ is defined by \eqref{1.8}.
It is interesting to notice that the reaction term is $u^p$ rather than ${\rm e}^{\mu t} u^p$ in \eqref{1.9}. This point is very different from the continuous cases addressed by \cite{Guiw} and \cite{WY}.

\begin{theorem}\label{t2.2}	Assume $D_\mu, D_\omega<\infty$, $\omega_{\min}>0$ and $\inf\limits_{V}\mu>0$. Let $G$ satisfy the condition CDE$\,'(n,0)$ {\rm(}cf. Definition {\rm\ref{d2.2})} and volume growth condition:\vspace{-2mm}
 \begin{enumerate}[leftmargin=4em]
\item[\bf(VG)]\, There exist positive constants $c_0, m, r_0$ such that $\mathcal{V}(x,r)\leq c_0 r^m$ for $x \in V$ and $r \geq r_0$.\vspace{-2mm}
\end{enumerate}
If there exists $\bar x\in V$ such that $\liminf\limits_{d\left(\bar x, x\right) \to\infty}\psi(x)>0$,
then $T_\lambda <\infty$ for any $\lambda>0$.	
\end{theorem}

\begin{theorem}\label{t2.3} Let $D_{\mu}<+\infty$, and there exist a vertex $\tilde x\in V$ and $\psi_{\infty}>0$ so that $G$ satisfies the eigenvalue condition:\vspace{-2mm}
\begin{enumerate}[leftmargin=4em]
 \item[\bf(EC)]\, For any given $\varepsilon>0$ and $\delta>0$, there exists a finite connected subset $\Omega\subset V$ satisfying $\lambda_1(\Omega)<\varepsilon$ such that $d\left(x,\tilde x\right)>\delta$ for all $x \in \Omega$,
where $\lambda_1(\Omega)$ is the smallest eigenvalue of the eigenvalue problem \eqref{2.1},\vspace{-1mm}
\end{enumerate} and that $\lim\limits_{d\left(\tilde x, x\right) \to\infty} \psi(x)=\psi_{\infty}$. Then we have the following conclusions:
	\begin{description}\vspace{-2mm}
\item {\rm(1)}\; $T_{\lambda}\le\dd\frac{1}{(p-1)\lambda^{p-1}\psi_{\infty}^{p-1}}, \;\;\forall\; \lambda>0$.\vspace{-2mm}
\item {\rm(2)} If $\psi(x)\le\psi_{\infty}$ for all $x\in V$, then
$\dd\lim_{\lambda\to 0} \lambda^{p-1}T_\lambda =\frac1{(p-1)\psi_{\infty}^{p-1}}$.\vspace{-2mm}
\item{\rm(3)} \; $\dd\frac1{(p-1)\|\psi\|_{\ell^{\infty}(V)}^{p-1}}\le\liminf_{\lambda\to 0}\lambda^{p-1} T_\lambda \le \limsup_{\lambda \to 0} \lambda^{p-1} T_\lambda \le\frac1{(p-1) \psi_{\infty}^{p-1}}$.
\end{description}
\end{theorem}

\begin{remark} There are many graphs that satisfy the condition {\bf(EC)}, see Examples {\rm 8.1-8.3} in {\rm\cite{BHua}} for example.
\end{remark}
\begin{remark}
Theorem \ref{t2.3} does not require any geometric assumptions about the graph $G$, but requires the eigenvalue condition {\bf(EC)}, which is different from the continuous case raised by \cite{Guiw} and \cite{WY}.
\end{remark}
\begin{theorem}\label{t2.4} Let $G=(V,E)$ be a finite connected graph. Then we have the following conclusions:
\begin{description}
	\item{\rm(1)} There exists $\Lambda=\Lambda(p,\psi,G)>0$ so that if $\lambda>\Lambda$, then $T_{\lambda}<\infty$. Moreover, \begin{equation*}
		\lim_{\lambda\to \infty}\lambda^{p-1}T_\lambda =\frac1{(p-1)\left(\max_V \psi\right)^{p-1}}.
	\end{equation*}
	\item{\rm(2)} If $\min\limits_{V}\psi(x)>0$, then $T_\lambda <\infty$ for $\lambda>0$.
	\item{\rm(3)} If $\min\limits_{V}\psi(x)>0$, then
	\begin{equation*}
		\frac1{(p-1)\big(\max_{V}\psi\big)^{p-1}}\leq \liminf_{\lambda\to 0}\lambda^{p-1} T_\lambda \leq \limsup_{\lambda \to 0}\lambda^{p-1}T_\lambda \leq\frac1{(p-1)\big(\min_{V}\psi\big)^{p-1}}.
	\end{equation*}
\end{description}
\end{theorem}
\begin{remark}
Comparing Theorem \ref{t2.3} and Theorem \ref{t2.4}, we find that the number of $V$ has a significant impact on the asymptotic behavior of $T_{\lambda}$.
\end{remark}

Yamauchi (\cite[Theorems~1]{Ya2011}) proved that if $N=1$ and $$\max\left\lbrace\liminf\limits_{x\to +\infty}\phi(x),\;\liminf\limits_{x\to-\infty}\phi(x)\right\rbrace>0,$$ then the
classical solution of \eqref{1.1} blows up in finite time, and the blow-up time is estimated as
\bess
T[\phi]\le\frac{1}{p-1}\left(\max\left\lbrace\liminf\limits_{x\to+\infty}
\phi(x),\;\liminf\limits_{x\to-\infty}\phi(x)\right\rbrace\right)^{1-p}.
    \eess
Yamauchi' ideas are not effective for problem \eqref{1.9}. We shall use a completely different approach to obtain a similar result as follows.

\begin{theorem}\label{t1.6}
Suppose $N=1$, $D_{\mu}<\infty$, $\lambda=1$ and $\alpha>0$. If
   $$\max\left\lbrace\liminf\limits_{\mathbb{Z}\ni x\to +\infty}\psi(x),\;\; \liminf\limits_{\mathbb{Z}\ni x\to-\infty}\psi(x)\right\rbrace >0,$$ 
then the solution of \eqref{1.9} blows up in finite time, and the blow-up time is estimated as
    \begin{equation}
T_1\left(\mathbb{Z},\Delta_{\mathbb{Z}}\right)\le\frac{1}{p-1}\left(\max\left\lbrace \liminf\limits_{\mathbb{Z}\ni x\to +\infty}\psi(x),\;\;\liminf\limits_{\mathbb{Z}\ni x\to-\infty}\psi(x)\right\rbrace\right)^{1-p}.
    \label{1.11}\end{equation}
\end{theorem}

This paper is organized as follows. In Section 2, we introduce some notations, concepts and known results that will be used frequently in the following sections. In Section 3, we give auxiliary results, including the discrete Phragm\'{e}n-Lindel\"{o}f principle of parabolic equations and upper and lower solutions method. Section 4 is devoted to the proofs of Theorems \ref{t2.1}-\ref{t2.3}. In Sections 5 and 6, we prove Theorems \ref{t2.4} and \ref{t1.6}, respectively.

\section{Preliminaries}

\subsection{The Laplacian on graphs}

Define the set of $\ell^p$ integrable functions on $V$ with respect to
the measure $\mu$ by
$$\ell^p(V,\mu)=\left\{f\in V^{\mathbb{R}}:\sum_{x \in V} \mu(x)|f(x)|^p<\infty\right\}, \;\; 1\leq p<\infty,$$
with the norm
 $$\|f\|_{\ell^p(V,\mu)}=\left(\sum_{x \in V} \mu(x)|f(x)|^p\right)^{1/p},\;\;\;f\in \ell^p(V,\mu).$$
For $p=\infty$, we define
  $$\ell^{\infty}{(V)}=\left\{f\in V^{\mathbb{R}}:\sup_{x\in V}|f(x)|<\infty\right\},$$
with the norm
$$\|f\|_{\ell^{\infty}{(V)}}=\sup_{x\in V} |f(x)|,\;\;\;f\in \ell^{\infty}{(V)}.$$

For any $H\in {V}^{\mathbb{R}}$, the graph Laplacian $\Delta$ on $G$ is defined by
\begin{equation*}
	\Delta H(x)=\frac1{\mu(x)}\sum_{y\in V:\, y\sim x}\omega_{x y}(H(y)-H(x)),~x\in V.
\end{equation*}
It follows from \cite{HAESELER} that
$D_\mu<\infty$ if and only if $\Delta$ is bounded on $\ell^p(V,\mu)$ for all $p\in [1,\infty]$. For $f,g\in V^{\mathbb{R}}$, the gradient form of $f$ and $g$ reads
\begin{equation*}
	\Gamma(f,g)(x)=\frac1{2\mu(x)}\sum_{y\in V:\,y\sim x}{\omega}_{x y}(f(y)-f(x))(g(y)-g(x)),
\end{equation*}
the iterated gradient form $\Gamma_2$ is defined by
\begin{equation*}
	 \Gamma_2(f,g)(x)=\frac12(\Delta \Gamma(f,g)-\Gamma(f,\Delta g)-\Gamma(\Delta f,g))(x).
\end{equation*}
We will write $\Gamma(f)=\Gamma(f,f)$ and $\Gamma_2(f)=\Gamma_2(f,f)$.

In this paper, we always assume that $\Omega$ is a finite connected subset of $V$. We define the boundary of $\Omega$ by
\begin{equation*}
\partial\Omega=\{x\in \Omega: \text{there exists}\, y\in \Omega^{c}~\text{such that}~y\sim x\},
\end{equation*}
and the interior of $\Omega$ by $\Omega^\circ=\Omega\setminus\partial\Omega$ is nonempty. For any $f \in {\ell^2}(\Omega; \mu)$, we can extend it to a function $\tilde{f} \in {\ell^2}(V ; \mu)$ by
\begin{equation*}
	\tilde{f}(x)=\left\{\!\!\begin{array}{ll}
		f(x), \;\;& x \in \Omega^\circ,\\[1mm]
		0, & x \notin \Omega^\circ.
	\end{array}\right.
  \end{equation*}
We define the Dirichlet Laplace operator $\Delta_{\Omega}: {\ell^2}(\Omega; \mu)\to {\ell^2}(\Omega; \mu)$ by $\Delta_{\Omega} f=(\Delta\tilde{f})\big|_{\Omega^\circ}$, i.e.,
 \begin{equation*}
	\Delta_{\Omega}f(x)=\frac{1}{\mu(x)}\sum_{y\in V}(\tilde{f}(y)-\tilde{f}(x))\omega_{y x}=\frac{1}{\mu(x)}\sum_{y\in V}(\tilde{f}(y)-f(x))\omega_{y x},\quad x\in \Omega^\circ.
	\end{equation*}
By \cite{CLC, gri09}, the eigenvalue problem
\begin{equation}\label{2.1}
	\left\{\!\!\begin{array}{ll}
	-\Delta_{\Omega} \phi=\lambda \phi,\;\;&x\in\Omega\setminus\partial\Omega,\\[1mm]
	\phi =0,\;\;&x\in\partial\Omega
	\end{array}\right.
\end{equation}
admits a sequence of eigenvalues $0<\lambda_1(\Omega)\leq\lambda_2(\Omega)\leq\cdots\leq\lambda_N(\Omega)$, where $N=\#\Omega^{\circ}$. Moreover the eigenfunction $\phi(x)$ corresponding to $\lambda_1(\Omega)$ could be chosen such that $\sum_{x\in\Omega}\phi(x)\mu(x)=1$ and $\phi(x)>0$ in $\Omega$.

\subsection{Curvature dimension condition and the heat kernel on graphs}\label{s2.2}

For a function $P:(0,\infty)\times V \times V\to \mathbb{R}$ and all bounded $u_0 \in V^{\mathbb{R}}$, if
  $$u(x,t)=\sum_{y\in V}P(t,x,y)u_0(y){\mu(y)},\quad x\in V,\;\;t>0$$
is differentiable in $t$, satisfies
\begin{equation}\label{2.2}
	u_t=\Delta u~\;\;\;\text{on}\;~G=(V,E),
\end{equation}
 and for all $x\in V$ and $t>0$,
$\lim\limits_{t\to0^+}u(x,t)=u_0(x)$ holds, we say that $P$ is a fundamental solution of \eqref{2.2}, i.e., $P(t,x,y)$ is the heat kernel on $G$.

Let us recall the curvature dimension conditions introduced in \cite{HLLY}.

\begin{definition}
Let $x\in V$, $n>0$ and  $K\in\mathbb{R}$. We call that a graph $G$ satisfies the exponential curvature dimension inequality $CDE(x,n,K)$ at the vertex $x$, if for any function $f:V\to (0,+\infty)$ satisfying $\Delta f(x)<0$, there holds:
 $$\Gamma_2(f)(x)-\Gamma\kk(f,\frac{\Gamma(f)}{f}\rr)(x)\geq\frac1{n}[(\Delta f)(x)]^2+K\Gamma(f)(x).$$
We say that $CDE(n,K)$ is satisfied if $CDE(x,n,K)$ is satisfied for all $x\in V$.
\end{definition}

\begin{definition}\label{d2.2} Let $x\in V$, $n>0$ and  $K\in\mathbb{R}$. A graph $G$ satisfies the exponential curvature dimension inequality $CDE'(x,n,K)$, if for any function $f:V\to \mathbb{R}^+$, we have
   $$\Gamma_2(f)(x)-\Gamma\kk(f,\frac{\Gamma(f)}{f}\rr)(x)\geq
   \frac1{n}[f(x)(\Delta \log f)(x)]^2+K\Gamma(f)(x).$$
We say that $CDE'(n,K)$ is satisfied if $CDE'(x,n,K)$ is satisfied for all $x\in V$.
\end{definition}

For convenience, we summarize some important properties of the heat kernel $P(t,x,y)$ on $G$ in the following propositions.

\begin{proposition}[\cite{AW,RW}]\label{p2.2} For $t,s>0$ and any $x,y\in V$, we have

{\rm(i)}~$P_t(t,x,y)=\Delta_xP(t,x,y)=\Delta_yP(t,x,y)$,\vspace{1mm}

{\rm(ii)}~$P(t,x,y)> 0$,\vspace{1mm}

{\rm(iii)}~$P(t,x,y)=P(t,y,x)$,\vspace{1mm}

{\rm(iv)}~$\sum_{y\in V}P(t,x,y)\mu(y)\leq 1$,\vspace{1mm}

{\rm(v)}~$\sum_{z\in V}P(t,x,z)P(s,z,y)\mu(z)=P(t+s,x,y)$.
\end{proposition}

\begin{proposition}[\cite{BHLLMY}] Suppose that $G$ satisfies $D_{\mu}, D_{\omega}<\infty$ and $CDE(n,0)$. Then there exists a positive constant $C_1=C_1(n, D_{\omega}, D_{\mu})$ such that, for any $x, y \in V$ and $t> 0$,
		\begin{equation*}
		P(t,x,y)\leq\frac{C_1}{\mathcal{V}(x,\sqrt{t})}.
	\end{equation*}
Furthermore, there exist $C_2=C_2(D_{\omega}, D_{\mu})>0$ and $C_3=C_3(n, D_{\omega}, D_{\mu})>0$ such that
 \begin{equation*}
P(t,x,y)\geq C_2\frac1{t^n}\exp{\Bigg(-C_3\frac{d^2(x,y)}{t-1}\Bigg)},\;\;t\ge 1.
	\end{equation*}
\end{proposition}

\begin{proposition}[\cite{HLLY}]\label{p2.3} Suppose $G$ satisfies $D_{\mu}, D_{\omega}< \infty$ and $CDE'(n, 0)$. Then there exist $C(n)>0$ and $c=c(n,D_{\omega},D_{\mu})>0$ such that
\begin{equation*}
	P(t,x,y) \geq\frac{C(n)}{\mathcal{V}(x,\sqrt{t})}\exp \left(-c\frac{d^2(x,y)}{t}\right),\;\;x, y\in V,\; t>t_0>0.
 \end{equation*}
In particular,
 \bess
	P(t,x,x)\geq\frac{C(n)}{\mathcal{V}(x,\sqrt{t})},\;\;x\in V,\;t> t_0>0.
 \eess
\end{proposition}

Let $\mathbb{Z}^{+}$ be the collection of all positive integers.
	
\begin{proposition}\label{p2.4} Suppose that $G=(V,E)$ is a locally finite connected graph. Let $T, \tilde{T}_1>0$ and $J: V\times[0,\tilde{T}_1]\to\mathbb{R}$. If $J$ is bounded in $V\times[0,\tilde{T}_1]$, then $\sum_{y\in V}\mu(y)J(y,s)P(t,x,y)$ converges uniformly w.r.t. $x\in V$, $t\in[0,T]$ and $s\in [0,\tilde{T}_1]$.
\end{proposition}

\begin{proof} Firstly, the following holds (the formula on the last line of page 12 of  \cite{HLLY}):
 \begin{equation}\label{2.3}
	\sum_{y\in V} \mu(y) P(t,x,y)J(y,s)=\sum_{n=0}^{\infty}\frac{t^n}{n!}\left(\Delta^n J\right)(x,s),\;\;~x\in V,~t\in[0,T],\; s\in [0,\tilde{T}_1].
\end{equation}
We shall estimate $\left(\Delta^n J\right)(x,s)$. Since $J$ is bounded, we may find constant $C_*>0$ such that $|J(y,s)|\le C_*$ for $y\in V$ and $s\in [0,\tilde{T}_1]$. Thus one has
\begin{equation*}
	\begin{aligned}
|\Delta J(x,s)|=\left|\sum_{y\sim x}[J(y,s)-J(x,s)]\frac{w_{y x}}{\mu(x)}\right|
\leq\sum_{y\sim x}2C_*\frac{w_{y x}}{\mu(x)}
\leq 2D_{\mu} C_*,\;\;x\in V,\;s\in [0,\tilde{T}_1].
	\end{aligned}
\end{equation*}
By induction, we can obtain that
\begin{equation*}
	\left|\Delta^n J(x,s)\right|\leq \left(2D_\mu\right)^n C_*,\;\;n\in\mathbb{Z}^{+},~x\in V,\;s\in [0,\tilde{T}_1].
\end{equation*}
It follows that, for $x\in V,t\in(0, T]$ and $s\in [0,\tilde{T}_1]$,
 \begin{equation*}
	\sum_{n=0}^{\infty}\left|\frac{t^n}{n !}\left(\Delta^n J\right)(x,s)\right|
\leq\sum_{n=0}^{\infty}t^n \frac{\left(2D_\mu\right)^n}{n!}C_*={\rm e}^{2D_{\mu} t}C_*\leq {\rm e}^{2D_\mu T}C_*.
  \end{equation*}
This shows that $\sum_{n=0}^{\infty}\frac{t^n}{n!}\left(\Delta^n J\right)(x,s)$ converges uniformly w.r.t. $x\in V,~t\in[0, T]$ and $s\in [0,\tilde{T}_1]$. Then, by \eqref{2.3}, the sum $\sum_{y\in V}P(t,x,y)J(y,s)\mu(y)$ converges uniformly for $x\in V,t\in[0, T]$ and $s\in [0,\tilde{T}_1]$.
\end{proof}

\section{Auxiliary result}{\setlength\arraycolsep{2pt}

\subsection{The maximum principle and comparison principle on graphs}

\begin{lemma}\label{c3.1} Assume $c(x,t)<0$ in $(\Omega\setminus\partial\Omega)\times (t_0,T]$, and $v$ satisfies
	\begin{equation*}
v_{t}-\Delta v-c(x,t)v\ge 0~\;\; \text{in}\;~ (\Omega\setminus\partial\Omega)\times(t_0,T].
  \end{equation*}
If $\min\limits_{\Omega\times[t_0,T]}v<0$, then $v$ can't reach its negative minimum in $(\Omega\setminus\partial\Omega)\times(t_0,T]$.
\end{lemma}

\begin{proof}	Suppose by way of contradiction that there exists $(x_1,t_1)\in(\Omega\setminus\partial\Omega)\times(t_0,T]$ such that
$v(x_1,t_1)=\min_{\Omega\times[t_0,T]}v<0$. Since $c>0$ in $(\Omega\setminus\partial\Omega)\times (t_0,T]$, by the definition of $\Delta$, we deduce that $(v_{t}-\Delta v-cv)\big|_{(x_1,t_1)}<0$, which is a contradiction.
\end{proof}

Next, we develop the discrete Phragm\'{e}n-Lindel\"{o}f principle of parabolic equations on infinite locally finite graph.

\begin{lemma}\label{l3.2} Let $D_{\mu}<\infty$, $G=(V,E)$ be an infinite locally finite graph and $T\in (0,\infty]$. Suppose that $c,z\in L^\infty(V\times I)$ for any bounded interval $I\subset [0,T)$, and that $z\in C_V([0,T))$, $z_t\in C_V(0,T)$. If $z$ satisfies
 \bess\left\{\begin{array}{ll}
 z_t-\Delta z-c(x,t)z\ge 0, \; &(x,t)\in V\times(0,T),\\[1mm]
 z(x,0)\ge 0, &x \in V,
 \end{array}\right.\eess
then $z\ge 0$ on $V\times(0,T)$.
 \end{lemma}

\begin{proof} Without loss of generality, we assume that $c\le-1$. Suppose on the  contrary that there exists $(i_0,\tau_0)\in V\times (0,T)$ such that $z(i_0,\tau_0)<0 $. Then we can find $t_0\in (0,\tau_0]$ such that
$z(i_0,t_0)=\min\limits_{[0,\tau_0]} z(i_0,\cdot)<0$, and then $z_t(i_0,t_0)\le 0$.

{\bf Step 1}. For the given $x\in V$, define
 \begin{equation*}
 	A_{x}=\{y\in V: y\sim x\} \cup \{x \}.
 \end{equation*}
Since $G$ is connected, by Lemma \ref{c3.1}, there exists $t_1 \in  (0,t_0]$ and $i_1 \sim i_0$ such that
 \bess
	z(i_1,t_1)=\min_{A_{i_0}\times[0,t_0]} z<z(i_0,t_0),\;\;\;\text{and}\;\;
	z_t(i_1,t_1) \le 0.
 \eess
If deg$(i_1)=1$, then by Lemma \ref{c3.1}, we get a contradiction. Thus deg$(i_1)>1$. By use of Lemma \ref{c3.1} again, there exists $i_2$ satisfying $i_0\not=i_2\sim i_1$ and $t_2\in (0,t_1]$ such that
\bess
	z(i_2,t_2)=\min\limits_{A_{i_0}\cup A_{i_1} \times [0,t_1]} z <z(i_1,t_1)<0,\;\;\;\text{and}\;\;
	z_t(i_2,t_2) \le 0.
 \eess
Similary, deg$(i_2)>1$.

Repeating this process, we can find three sequences $\{i_n\}$, $\{t_n\}$ and $\{z(i_n,t_n)\}$ satisfying
 \bes
	&i_1\sim i_2 \sim\cdots\sim i_n\sim i_{n+1}\sim \cdots,\;\;\;
 \text{deg}(i_n)\ge 2,&\nm\\[1mm]
 &\label{3.1}
	0<t_n\le t_{n-1}\le t_0,\;\;  z(i_n,t_n)<z(i_{n-1},t_{n-1})<0&
 \ees
and
 \bes\label{3.2}	 z(i_n,t_n)=\min_{(\cup_{k=0}^{n-1}A_{i_k})\times[0,\,t_{n-1}]}z(x,t),\;\;\;z_t(i_n,t_n) \le 0,\,~ n=1,2,3,\cdots.
\ees
Since $(z_t-\Delta z-cz)|_{(i_n,t_n)}\ge 0$ and
$0\ge z_t(i_n,t_n)$, it follows that $-(cz)(i_n,t_n)\ge \Delta z(i_n,t_n)$.
In view of $c\le -1$ and \eqref{3.1}, it yields
 \begin{equation}\label{3.3}
	0>z(i_1,t_1)\ge \Delta z(i_n,t_n).
 \end{equation}
Since $z(x,t)$ and $c(x,t)$ are bounded, we deduce that there exists constant $c_1>0$ such that
\begin{equation}\label{3.4}
	|c(x,t)|,\; |z(x,t)| \le c_1~\;\text{for}\,~x \in V,~ t\in [0,t_0].
 \end{equation}
Thanks to $\sum_{y\in V:\,y\sim x}\frac{\omega_{xy}}{\mu(x)}\le D_{\mu}$, it follows that
\begin{equation}\label{3.5}
	\begin{split}
		|\Delta z(i_{n+1},t)|=\left|\sum_{y\in V:\,y\sim i_{n+1}}[z(y,t)-z(i_{n+1},t)]\frac{\omega_{yi_{n+1}}}{\mu(i_{n+1})}\right|
		&\le 2c_1D_{\mu}.
	\end{split}
\end{equation}

{\bf Step 2}. Rewriting
 \bes\label{3.6}
\Delta z(i_n,t_n)&=&\sum_{y\in V:\,y\sim i_n}\left[z(y,t_n)-z(i_n,t_n)\right]\frac{\omega_{y i_n}}{\mu(i_n)}\nm\\[1mm]
&=&\left(\sum_{y=i_{n+1}}+\sum_{y\sim x,\,y\not =i_{n+1},\,i_{n-1}}+\sum_{y=i_{n-1}}\right)\left[z(y,t_n)-z(i_n,t_n)\right] \frac{\omega_{y i_n}}{\mu(i_n)}. \ees
Throughout this step we assume that $y\sim i_n$. It follows from \eqref{3.5} that if $y=i_{n+1}$ then
  \bes\label{3.7}
 z(y,t_n)-z(i_n,t_n)	 &=&z(i_{n+1},t_n)-z(i_{n+1},t_{n+1})+z(i_{n+1},t_{n+1})-z(i_n,t_n)\nm\\[1mm]
 &=& \int_{t_{n+1}}^{t_n} z_t(i_{n+1},t){\rm d}t+ z(i_{n+1},t_{n+1})-z(i_n,t_n)\nm\\[1mm]
 &\ge& -\int_{t_{n+1}}^{t_n} \left| \Delta z(i_{n+1},t)+ (cz)(i_{n+1},t)\right|{\rm d}t -\left| z(i_{n+1},t_{n+1})-z(i_n,t_n)\right|\nm\\[1mm]
 &\ge& -\int_{t_{n+1}}^{t_n} 2c_1D_{\mu}+c_1^2 {\rm d}t -\left| z(i_{n+1},t_{n+1})-z(i_n,t_n)\right|\nm\\[1mm]
 &=&-(2c_1D_{\mu}+c_1^2)(t_n-t_{n+1})-\left| z(i_{n+1},t_{n+1})-z(i_n,t_n)\right|\nm\\[1mm]
 &=&:A_n.
 \ees
If $y\not= i_{n+1}$ and $y\not = i_{n-1}$, then $z(y,t_{n+1})\ge z(i_{n+1},t_{n+1})$ by \eqref{3.1} and \eqref{3.2}, and
 \bess
|\Delta z(y,t)| =\left|\sum_{x\in V:\,x\sim y} [z(x,t)-z(y,t)]\frac{\omega_{xy}}{\mu(y)}\right|
\le 2c_1\sum_{x\in V}\frac{\omega_{xy}}{\mu(y)}
\le 2c_1 D_{\mu}
 \eess
by \eqref{3.4}. Thus we have
\bes\label{3.8}
z(y,t_n)-z(i_n,t_n)&=&z(y,t_n)-z(y,t_{n+1})+z(y,t_{n+1})-z(i_{n+1},t_{n+1})
+z(i_{n+1},t_{n+1})-z(i_n,t_n)\nm\\
&\ge& \int_{t_{n+1}}^{t_n} z_t(y,t){\rm d}t+ z(i_{n+1},t_{n+1})-z(i_n,t_n)\nm\\
&\ge& \int_{t_{n+1}}^{t_n} \Delta z(y,t)+ (cz)(y,t) {\rm d}t + z(i_{n+1},t_{n+1})-z(i_n,t_n)\nm\\
&\ge& -\int_{t_{n+1}}^{t_n} \left| \Delta z(y,t)+(cz)(y,t)\right|{\rm d}t-\left| z(i_{n+1},t_{n+1})-z(i_n,t_n)\right|\nm\\
&=&-(2c_1D_{\mu}+c_1^2)(t_n-t_{n+1})-\left| z(i_{n+1},t_{n+1})-z(i_n,t_n)\right|\nm\\
&{=}&A_n.
\ees
If $y=i_{n-1}$, then
  \bes\label{3.9}
z(y,t_n)-z(i_n,t_n)	&=&z(i_{n-1},t_n)-z(i_n,t_n)\nm\\	 &=&z(i_{n-1},t_n)-z(i_{n-1},t_{n-1})+z(i_{n-1},t_{n-1})-z(i_n,t_n)\nm\\
		&=&\int_{t_{n-1}}^{t_n} z_t(i_{n-1},t){\rm d}t+ z(i_{n-1},t_{n-1})-z(i_n,t_n)\nm\\
		&\ge& \int_{t_{n-1}}^{t_n}  \Delta z(i_{n-1},t)+ (cz)(i_{n-1},t) {\rm d}t + z(i_{n-1},t_{n-1})-z(i_n,t_n) \nm\\
		&\ge& -\int_{t_{n-1}}^{t_n} \left| \Delta z(i_{n-1},t)+ (cz)(i_{n-1},t)\right|{\rm d}t -\left|z(i_{n-1},t_{n-1})-z(i_n,t_n)\right|\nm\\
&=&-(2c_1D_{\mu}+c_1^2)(t_n-t_{n-1})-\left| z(i_{n-1},t_{n-1})-z(i_n,t_n)\right|\nm\\
	&=:&B_n.
	\ees
It follows from \eqref{3.6}-\eqref{3.9} that
\bes\label{3.10}
\Delta z(i_n,t_n)&=&\left(\sum_{y=i_{n+1}}+\sum_{y\sim x,\,y\not=i_{n+1},\,i_{n-1}}+\sum_{y=i_{n-1}}\right)\left[z(y,t_n)-z(i_n,t_n)\right]\frac{\omega_{y i_n}}{\mu(i_n)}\nm\\[1mm]
&\ge& A_n\frac{\omega_{i_{n+1} i_n}}{\mu(i_n)}+ A_n\sum_{y\sim x,\, y\not=i_{n-1},\,i_{n+1}}\frac{\omega_{y i_n}}{\mu(i_n)}+B_n\frac{\omega_{i_{n-1} i_n}}{\mu(i_n)}\nm\\[1mm]
	&\ge& -(2|A_n|+|B_n|)D_{\mu}.
	\ees

{\bf Step 3}. By \eqref{3.1}, $\underline{t}=\lim\limits_{n\to \infty}t_n$ is well defined, and we can define $A=\lim\limits_{n\to \infty} z(i_n,t_n)\in \mathbb{R}$ since $z(x,t)$ is bounded. It follows that
$\lim\limits_{n\to \infty}(2|A_n|+|B_n|)=0$ by the expressions of $A_n$ and $B_n$. Letting $n\to \infty$ in \eqref{3.10}, we get a contradiction with \eqref{3.3}.\end{proof}

In order to prove our main results, we need an auxiliary result, which provides a representation to the solution of \eqref{1.6} via
an integral equation with the heat kernel.

\begin{theorem}{\rm(\cite[Theorem 3.1]{Wc})}\label{t3.1} Suppose that $D_\mu<\infty$. Let $T_\lambda $ be the life span of \eqref{1.6} defining by \eqref{1.8}. Then for any $0<T<T_\lambda $, the solution $u$ of \eqref{1.6} satisfies
 \begin{equation}\label{3.11}
u(x,t)=\lambda\sum_{y\in V}P(t,x,y)\psi(y)\mu(y)+\int_0^t\sum_{y\in V} P(t-s,x,y)u^p(s,y)\mu(y) {\rm d}s, \;\;x\in V,\; 0<t<T.\;\;
\end{equation}
\end{theorem}

\begin{lemma}[\cite{KL}]\label{l3.3} If $D_\mu<\infty$, then
 $\sum_{y\in V}\mu(y)P(t,x,y)=1$ for $x\in V$ and $t>0$. \end{lemma}

\subsection{Upper and lower solution method on infinite locally finite graphs}
{\setlength\arraycolsep{2pt}

Throughout this subsection, we assume that $G=(V,E)$ is an infinite locally finite graph. Let $T>0$ and consider the following problem
\begin{equation}\label{3.12}
	\begin{cases}
		u_t-\Delta u= f(x,t,u),~&(x,t)\in V\times(0,T],\\[1mm]
		u(x,0)= \psi(x),~&x\in V.	
	\end{cases}	\vspace{-2mm}
\end{equation}		
	
\begin{definition} Suppose that $u\in L^\infty(V\times(0,T])\cap C_V([0,T])\cap C^1_V((0,T])$. If $u$ satisfies
	\begin{equation*}
		\begin{cases}
			u_t-\Delta u\ge (\le) f(x,t,u),~&(x,t)\in V\times (0,T],\\
			u(x,0)\ge (\le) \psi,~&x\in V,		
		\end{cases}	
	\end{equation*}
then we call $u$ an upper $($a lower$)$ solution of \eqref{3.12}.
\end{definition}

\begin{lemma}{\rm (Ordering of upper and lower solutions)}\label{l3.4}
Suppose $D_{\mu}<\infty$. Let $\bar{u}$, $\underline{u}$ be the bounded upper and lower solutions of \eqref{3.12}, respectively. Denote ${\underline{\eta}}=\inf\limits_{V\times[0,T]}\min\{\bar{u},\underline{u}\}$ and ${\bar\eta}=\sup\limits_{V\times[0,T]}\max\{\bar{u},\underline{u}\}$. If $f$ satisfies the Lipschitz condition in $u\in{[\underline{\eta},\bar\eta]}$, i.e., there exists a constant $M>0$ such that
 \begin{equation}\label{3.13}
		|f(x,t,u)-f(x,t,v)| \le M |u-v|,~\forall~(x,t)\in V\times[0,T],~u,v\in {[\underline{\eta},\bar\eta]},
	\end{equation}then $\bar{u}\ge \underline{u}$ in ${V}\times [0,T]$.
\end{lemma}

\begin{proof}Setting
	\begin{equation*}
	{\hat{c}(x, t)=\int_0^1f_u(x,t,\underline{u}(x, t)+s(\bar{u}(x, t)-\underline{u}(x, t))){\rm d}s,}
	\end{equation*}	
and $u=\bar{u}-\underline{u}$. Then $u$ satisfies
	\begin{equation*}\begin{cases}
u_{t}-{\Delta u}- \hat{c}(x,t)u(x,t)\ge 0,~~&(x,t)\in V\times(0,T],\\
	u(x,0)\ge 0,~&x\in V.
\end{cases}	\end{equation*}
Clearly, $\hat{c}(x,t)$ is bounded. It follows from Lemma \ref{l3.2} that $u(x,t)\ge 0$ for $x\in {V}$ and $t\in [0,T]$.
\end{proof}

\begin{theorem}\label{t3.2}
Assume {$D_{\mu}<\infty$}. Let $\bar u$, $\underline{u}$ be the bounded upper and lower solutions of \eqref{3.12}, respectively, and $0\le,\not\equiv \psi\in V^{\mathbb{R}}$ be given. Assume that
	\begin{equation}\label{3.14}
f(x,t,0)=0,\;\;\text{and}\;\;f(x,\cdot, u)\in C([0,T])\;\;\;\text{for}\;\;
x\in V, \;t\in[0,T],\;u\in [\underline{\eta},\bar{\eta}],
	\end{equation}
and $f$ satisfies the Lipschitz condition \eqref{3.13} in $u\in[\underline{\eta},\bar{\eta}]$, where $\underline{\eta}, \bar{\eta}$ are defined in Lemma {\rm\ref{l3.4}}.
Then \eqref{3.12} admits a unique solution $u\in C_V([0,T])\cap C^1_V((0,T])$ satisfying $\underline{u}\le u\le\bar u$ on $V\times[0,T]$.
\end{theorem}

\begin{proof} {\bf Step 1}. Define
 \begin{equation*}\label{dy}
\langle\underline{u},\bar{u}\rangle=\big\{u\in C_V([0,T]):\,\underline{u}\le u \le \bar{u}\big\}.
 \end{equation*}	
Then for any given $v\in \langle\underline{u},\bar{u}\rangle$, using \eqref{3.14}, we see that
 \begin{equation}\label{3.15}
H(x,t)=f(x,t,v(x,t))+Mv(x,t)\in C_V([0,T]).
 \end{equation}
Recalling that $v\in \langle\underline{u},\bar{u}\rangle$, we know that $v$ is bounded. Thus by \eqref{3.14} and \eqref{3.15}, we know that there exists constant $C_{H}=C_{H}(T)>0$ such that
\begin{equation}\label{3.16}
|H(x,t)|\le C_{H},~\;\;(x,t)\in V\times[0,T].
\end{equation}

{\bf Step 2}. It will be shown that the linear problem
\begin{equation}\label{3.17}
\left\{\!\begin{array}{ll}
u_{t}-\Delta u+Mu=H(x,t),~&x\in V,\; t\in(0, T],\\[1.5mm]
u(x,0)=\psi(x), &x \in V
\end{array}\right.
 \end{equation}
admits a unique bounded solution $u\in C_V([0,T])\cap C^1_V((0,T])$. Let $v={\rm e}^{M t}u$, then $v$ satisfies
\begin{equation}\label{3.18}
\begin{cases}v_t-\Delta v={\rm e}^{Mt}H(x, t), &x\in V,\; t\in(0, T],\\[1mm] v(x,0)=\psi(x), &x\in V.\end{cases}
\end{equation}

(i)\, Let \begin{equation*}
w(x,t)=\int_0^t\sum_{y\in V}P(t-s,x,y){\rm e}^{Ms}H(y,s)\mu(y){\rm d}s,\;\; x\in V,\; t\in[0,T].
\end{equation*}
We shall show that $w$ satisfies
  \bes
  w_t={\rm e}^{Mt}H(x,t)+\Delta w.
  \label{3.19}\ees
For any $\delta>0$. Direct calculations yield that
 \bes
w(x,t+\delta)-w(x,t)&=&\int_t^{t+\delta}\sum_{y\in V}P(t+\delta-s,x,y){\rm e}^{M s}H(y,s)\mu(y){\rm d}s\nm\\
&&+\int_0^t\sum_{y\in V}[P(t+\delta-s,x,y)-P(t-s,x,y)]{\rm e}^{Ms}H(y,s)\mu(y){\rm d}s.\quad\;\;\label{3.20}
	\ees
Clearly, $|H(x,s){\rm e}^{Ms}|\le {\rm e}^{Mt}C_{H}$ for $x\in V$ and $s\in [0,t]$ by \eqref{3.16}. Hence, $\sum_{y\in V}P(t+\delta-s,x,y){\rm e}^{Ms} H(y,s)\mu(y)$ converges uniformly w.r.t. $s\in[0,t]$ by Proposition \ref{p2.4}, and
$\sum_{y\in V}P(t+\delta-s,x,y){\rm e}^{Ms} H(y,s)\mu(y)$ is continuous w.r.t. $s\in[0,t]$ by \eqref{3.15}. Thus, by the first mean value theorem of integrals,
 \bes
\lim_{\delta\to 0}\frac 1\delta\int_t^{t+\delta}\sum_{y\in V}P(t+\delta-s,x,y){\rm e}^{M s}H(y, s)\mu(y){\rm d}s&=&\sum_{y\in V} P(0,x,y){\rm e}^{M t}H(y,t)\mu(y)\nm\\[1mm]
 &=&{\rm e}^{Mt}H(x,t).\label{3.21}
	\ees
Recalling the definition of the heat kernel $P(t,x,y)$, it is easy to see that $P_t(t,x,y)$ is continuous in $t$. Similar to the above,
 \bes
 &&\lim_{\delta\to 0}\frac 1\delta\int_0^t\sum_{y\in V}[P(t+\delta-s, x, y)-P(t-s,x,y)]{\rm e}^{Ms}H(y, s)\mu(y){\rm d}s\nm\\
 &=&\int_0^t\sum_{y\in V}P_t(t-s,x,y){\rm e}^{Ms}H(y,s)\mu(y){\rm d}s\nm\\
 &=&\int_0^t\sum_{y\in V
 }\Delta_x P(t-s,x,y){\rm e}^{Ms}H(y,s)\mu(y){\rm d}s
 \label{3.22}\ees
by Proposition \ref{p2.2}(i).

On the other hand, in view of Proposition \ref{p2.2}(iv) and \eqref{3.16}, it can be deduced that
 \bess
&&\sum_{\substack{z\in V\\ z\sim x}}\sum_{y\in V}\left|[P(t-s,z,y)-P(t-s,x,y)]\frac{\omega_{zx}}{\mu(x)}{\rm e}^{Ms}H(y,s)\mu(y)\right|\nm\\
&\le&\sum_{\substack{z\in V\\ z\sim x}}\sum_{y\in V}P(t-s,z,y)\frac{\omega_{zx}}{\mu(x)} {\rm e}^{Ms}|H(y,s)|\mu(y)
+\sum_{\substack{z\in V\\ z\sim x}}\sum_{y\in V}P(t-s,x,y)\frac{\omega_{zx}}{\mu(x)}{\rm e}^{Ms}|H(y,s)|\mu(y)\nm\\
&\le&\sum_{\substack{z\in V\\ z\sim x}} 2C_H\frac{\omega_{z x}}{\mu(x)} {\rm e}^{Ms}\nm\\
&=&2C_H\frac{m(x)}{\mu(x)}{\rm e}^{Ms}.
 \eess
Certainly,
\bess
 \int_0^t\sum_{\substack{z\in V\\ z\sim x}}\sum_{y\in V}\left|[P(t-s,z,y)-P(t-s,x,y)]\frac{\omega_{zx}}{\mu(x)}{\rm e}^{Ms}H(y,s)\mu(y)\right|{\rm d}s\le 2C_H\frac{m(x)}{\mu(x)}\int_0^t{\rm e}^{Ms}{\rm d}s.
 \eess
Thus we have
 \bess
\sum_{y\in V}\Delta_x P(t-s,x,y){\rm e}^{Ms}H(y,s)\mu(y)
&=&\sum_{y\in V}\sum_{\substack{z\in V\\ z\sim x}}[P(t-s,z,y)-P(t-s,x,y)]{\rm e}^{M s}H(y,s)\mu(y)\frac{\omega_{z x}}{\mu(x)}\nm\\
		&=&\sum_{\substack{z\in V\\ z\sim x}}\sum_{y\in V}[P(t-s,z,y)-P(t-s,x,y)]{\rm e}^{M s}H(y,s)\mu(y)\frac{\omega_{z x}}{\mu(x)},
	\eess
and
 \bes
&&\int_0^t\sum_{y\in V}\Delta_x P(t-s,x,y){\rm e}^{Ms}H(y,s)\mu(y)\ds\nm\\
&=&\sum_{\substack{z\in V\\ z\sim x}}\int_0^t\sum_{y\in V}[P(t-s,z,y)-P(t-s,x,y)]{\rm e}^{M s}H(y,s)\mu(y)\frac{\omega_{z x}}{\mu(x)}\ds\nm\\
&=&\Delta_x \int_0^t\sum_{y\in V}P(t-s,x,y){\rm e}^{Ms}H(y,s)\mu(y)\ds\nm\\
&=& \Delta w(x,t).
 \label{3.23}\ees
It follows from \eqref{3.20}-\eqref{3.23} that \eqref{3.19} holds.

(ii)\, Set
  $$v=\sum_{y\in V} P(t,x,y)\psi(y)\mu(y)+w.$$
Then $v$ satisfies
\bess
	v_t=\Delta v+{\rm e}^{Mt}H(x,t),\quad x\in V,\; t\in(0,T]
\eess
and $v(x,0)=\psi(x), x\in V$. That is, $v$ is a solution to \eqref{3.18}. So, $u={\rm e}^{-M t}v$ is a bounded solution of \eqref{3.17}. The uniqueness of bounded solutions of \eqref{3.17} is derived by Lemma \ref{l3.2}.
	
{\bf Step 3}. Define $\mathscr{F}:\langle\underline{u},\bar{u}\rangle\to C_V([0,T])$ by $u=\mathscr{F}(v)$ and
	\begin{equation*} \underline{u}_0=\underline{u},~\underline{u}_k=\mathscr{F}(\underline{u}_{k-1}),~ \bar{u}_0=\bar{u},~\bar{u}_k=\mathscr{F}(\bar{u}_{k-1}),~k=1,2,\cdots.
	\end{equation*}
We first prove that ${\mathscr F}$ is monotonically increasing, i.e., $u_1, u_2\in\langle\underline{u},\bar{u}\rangle$ with $u_1\le u_2$ implies ${\mathscr F}(u_1)\leq {\mathscr F}(u_2)$. Let $w_i=\mathscr{F}(u_i)$, and $w=w_2-w_1$. By virtue of \eqref{3.15}, we see that $w$ satisfies
 \begin{equation*}
 \left\{\begin{array}{ll}
w_{t}-\Delta w+Mw=f(x,t,u_2)-f(x,t,u_1)+M(u_2-u_1)\ge 0,\;&(x,t)\in V\times(0, T],\\[1.5mm]
w(x,0)=0,~&x\in V.
\end{array}\right.
	\end{equation*}
By Lemma \ref{l3.2}, $w\ge 0$, i.e., $w_1\le w_2$. Then we claim that $\underline{u}\le\underline{u}_1$. In fact, set $w=\underline{u}_1-\underline{u}$. Then $w$ satisfies
	\begin{equation*}
\left\{\!\!\begin{array}{ll}
w_{t}-\Delta w+Mw\ge 0,\;&(x,t)\in V\times(0, T],\\[1mm]
w(x,0)=0,&x \in V,
\end{array}\right.
	\end{equation*}
and by Lemma \ref{l3.2}	we get $w\ge 0$, i.e., $\underline{u}\le\underline{u}_1$. Similarly, $\bar{u}_1\le\bar{u}$. Using these facts and the inductive process we can show  that
 \bess
\underline{u}\le\underline{u}_1\le\cdots \le\underline{u}_k\le \bar{u}_k\le\cdots\le\bar{u}_1\le\bar{u},\;\;\forall\; k\ge 1.\eess

{\bf Step 4}. Now, we may define $\tilde{u}=\lim\limits_{k\to \infty}\underline{u}_k$ and $\hat{u}=\lim\limits_{k\to \infty}\bar{u}_k$.
Then $\tilde{u}, \hat u\in\langle\underline{u},\bar{u}\rangle$.	We shall show that $\tilde{u}, \hat u$ are solutions of \eqref{3.12}.

For any given $x\in V$ and any $(t_1,t)\subset(0,T]$, it follows from $\bar{u}_k=\mathscr{F}(\bar{u}_{k-1})$ that
	\begin{equation*}
\bar{u}_k(x,t)-\bar{u}_k(x,t_1)=\int_{t_1}^{t}\left[\Delta \bar{u}_k(x,s)-M\bar{u}_k(x,s)+f(x,s,\bar{u}_{k-1}(x,s))
+M\bar{u}_{k-1}(x,s)\right]\ds.
	\end{equation*}
Noticing that $\Delta \bar{u}_k-M\bar{u}_k+f(\cdot,\cdot,\bar{u}_{k-1})
+M\bar{u}_{k-1}$ is bounded. In view of the Lebesgue dominated convergence theorem, it follows by letting $k\to \infty$ in above equation that
 \bes\label{3.24}
\hat{u}(x,t)-\hat{u}(x,t_1)&=&\int_{t_1}^{t} \left[\Delta\hat{u}(x,s)-M\hat{u}(x,s)+f(x,s,\hat{u}(x,s))+M\hat{u}(x,s)\right]\ds \nm\\
&=&\int_{t_1}^{t}[\Delta\hat{u}(x,s)+f(x,s,\hat{u}(x,s)]\ds.	
	\ees
As $\hat{u}$ is bounded, there exist constants $C_1, C_2>0$ such that
\begin{equation*}
	\left| \Delta\hat{u}(x,s)+f(x,s,\hat{u}(x,s))\right| \le 2 C_1\frac{m(x)}{\mu(x)}+ C_2\le 2C_1D_{\mu}+C_2,\;\;\forall\; t_1\le s\le T.
\end{equation*}
This combines with \eqref{3.24} shows that $\hat{u}(x,t)$ is continuous with respect to  $t\in[t_1,T]$. Thus, by \eqref{3.24} and the arbitrariness of $t_1$ and $t$, we have that
 \begin{equation*}
 \hat{u}_{t}(x,t)-\Delta\hat{u}(x,t)=f(x,t,\hat{u}(x,t)),\;\;t\in(0,T].
\end{equation*}
Since $\bar{u}_k(x,0)=\psi(x)$, it is obvious that $\hat{u}(x, 0)=\psi(x)$. Therefore, $\hat{u}\in C_V([0,T])\cap C^1_V((0,T])$ solves \eqref{3.12}. Similarly, $\tilde{u}\in C_V([0,T])\cap C^1_V((0,T])$ solves \eqref{3.12}.
	
{\bf Step 5}. The uniqueness of solutions of \eqref{3.12} locating in  $\langle\underline{u},\bar{u}\rangle$ can be deduced by Lemma \ref{l3.4}.
\end{proof}

\section {Proofs of Theorems \ref{t2.1}-\ref{t2.3}}

\subsection{Proof of Theorem \ref{t2.1}}

The proof of Theorem \ref{t2.1} consists of the following lemmas.

\begin{lemma}\label{l4.1} Let $\Omega\subset V$ be finite and $\psi(x)\not\equiv 0$ on $\Omega$. If $u$ is the global solution of the problem \eqref{1.6}, then
 \begin{equation}\label{4.1}	\int_{\Omega}\lambda\psi\phi{\rm d}\mu\le\lambda_1^{\frac1{p-1}},
\end{equation}
where $\lambda_1$ is the smallest eigenvalue of the eigenvalue problem \eqref{2.1} and $\phi>0$ is the normalized eigenfunction corresponding to $\lambda_1$.
 \end{lemma}

\begin{proof} {\bf Step 1}. Let
  \begin{equation}\label{4.2}
	\eta(t)=\int_{\Omega}u(\cdot,t)\phi(\cdot){\rm d}\mu=\sum_{x\in\Omega}u(x,t)\phi(x)\mu(x),~t\ge 0.
	\end{equation} Then
	\begin{equation}\label{4.3}
	\eta'(t)=\int_{\Omega}u_{t}(\cdot,t)\phi(\cdot){\rm d}\mu=\int_{\Omega} \left[\Delta u(\cdot,t)+u^{p}(\cdot,t)\right] \phi(\cdot){\rm d}\mu,~t>0.
	\end{equation}
For clarity and simplicity, we denote $u_\Omega=u|_{\Omega}$ in the following.

{\bf Step 2}. We claim that
 \begin{equation}\label{4.4}
\phi(x)\Delta u(x,t)\ge\phi(x)\Delta_{\Omega}(u_{\Omega})(x,t),~~x\in\Omega,\; t\ge0.
 \end{equation}
In fact, since
 \bess
u(y,t)\left\{\begin{array}{ll}
\ge 0=\widetilde{u_{\Omega}}(y,t)~&\text{when}~y\not\in\Omega,\\[2mm]
=\widetilde{u_{\Omega}}(y,t)~&\text{when}~y\in\Omega,
	\end{array}\right.
 \eess
and $\phi(x)>0$ on $\Omega$, we have that, for all $x\in\Omega$,
\begin{equation*}
	\label{th4.1:4}
	\begin{split}
\phi(x)\Delta u(x,t)&=\phi(x)\sum_{y\in V:\,y\sim x}\left[u(y,t)-u(x,t)\right]\frac{\omega_{yx}}{\mu(x)}\\
&\ge\phi(x)\sum_{y\in V:\,y\sim x}\left[\widetilde{u_{\Omega}}(y,t)-\widetilde{u_{\Omega}}(x,t)\right]
\frac{\omega_{yx}}{\mu(x)}\\
&=\phi(x)\Delta_{\Omega}u_{\Omega}(x,t).
	\end{split}
\end{equation*}

{\bf Step 3}. Noticing that $\Delta_{\Omega}$ is self-adjoint and $\int_{\Omega} \phi {\rm d}\mu=1$. In view of \eqref{4.3}, \eqref{4.4} and Jensen's inequality, it deduces that
\begin{equation*}\begin{split}
	\eta'(t)&\geq \int_{\Omega}\phi(\cdot)\Delta_{\Omega}u_{\Omega}(\cdot,t){\rm d}\mu +\int_{\Omega}\phi(\cdot)u^{p}(\cdot,t){\rm d}\mu\\
	&=\int_{\Omega}(u_{\Omega})(\cdot,t)\Delta_{\Omega}\phi(\cdot){\rm d}\mu+\int_{\Omega} \phi(\cdot)u^{p}(\cdot,t){\rm d}\mu\\
	&=-\lambda_1\int_{\Omega}u(\cdot,t)\phi(\cdot){\rm d}\mu+\int_{\Omega}u^{p}(\cdot,t)\phi(\cdot){\rm d}\mu\\
	&\ge -\lambda_1\int_{\Omega}u(\cdot,t)\phi(\cdot){\rm d}\mu+\left(\int_{\Omega}u(\cdot,t)\phi(\cdot){\rm d}\mu\right)^{p}\\
	&=-\lambda_1\eta(t)+\eta^{p}(t),~t>0.
	\end{split}\end{equation*}
Set $\xi(t)=\eta(t){\rm e}^{\lambda_1t}$. Then
\bes\label{4.5}
	\xi'(t)&=&{\rm e}^{\lambda_1t}\left(\eta'(t)+\lambda_1\eta(t)\right)\nm\\[1mm]
	&\ge& {\rm e}^{\lambda_1t}\eta^{p}(t) \nm\\[1mm]
	&=&{\rm e}^{\lambda_1t}{\rm e}^{-\lambda_1pt}\xi^{p}(t) \nm\\[1mm]
	 &=&{\rm e}^{(1-p)\lambda_1t}\xi^{p}(t),~t>0.
\ees
Since $\lambda\psi(x)\ge,\not\equiv 0$ on $V$, we conclude that
$\eta(0)= \int_{\Omega}\lambda\psi\phi{\rm d}\mu>0$,
and hence $\xi(0)=\eta(0)>0$. Thus, by \eqref{4.5}, $\xi(t)>0$ for $t>0$. This combined with \eqref{4.5} allows us to deduce that $\frac{\xi'(t)}{\xi^{p}(t)}\ge{\rm e}^{(1-p)\lambda_1t}$, which implies
\begin{equation}\label{4.6}
0<\xi^{1-p}(t)\le\xi^{1-p}(0)+\frac{{\rm e}^{(1-p)\lambda_1t}-1}{\lambda_1}.
\end{equation}

{\bf Step 4}. Suppose by way of contradiction that
 \begin{equation*}
\lambda\int_{\Omega}\psi\phi{\rm d}\mu>\lambda_1^{\frac1{p-1}}.
 \end{equation*}
Since $\xi(0)=\eta(0)=\lambda\int_{\Omega}\psi\phi {\rm d}\mu>0$,
this implies that $1-\lambda_1\xi^{1-p}(0)>0$. Thus $$t_{*}=\frac{\text{ln}[1-\lambda_1\xi^{1-p}(0)]}{(1-p)\lambda_1}>0.$$
Since \begin{equation*}
	\lim\limits_{t\to t_*^{+}}\left(\xi^{1-p}(0)+\frac{{\rm e}^{(1-p)\lambda_1 t}-1}{\lambda_1}\rr)=0,
\end{equation*}
it is derived by \eqref{4.6} that $\lim_{t\to t_*^{+}}\xi(t)=\infty$,
and hence $\lim\limits_{t\to t_*^{+}}\eta(t)=\lim\limits_{t\to t_*^{+}}\xi(t){\rm e}^{-\lambda_1t}=\infty$.

On the other hand,
 \bess
\eta(t)\leq\int_{\Omega}\|u(\cdot,t)\|_{\ell^{\infty}(V)}\phi(x){\rm d}\mu
 =\|u(\cdot,t)\|_{\ell^{\infty}(V)},\;~t\ge 0,
	\eess
this implies $\lim_{t\to t_*^{+}}\|u(\cdot,t)\|_{\ell^{\infty}(V)}=\infty$.
This is a contradiction as $u$ is a global solution. Therefore, \eqref{4.1} holds.
\end{proof}

\begin{lemma}\label{l4.2} Suppose $D_{\mu}<+\infty$. Let $t_i>0$, $u_i\in C_V([0, t_i))\cap C^1_V((0,t_i))$ and $u_i\in L^\infty(V\times[0, t'_i])$ for any given  $0<t'_{i}<t_{i}$,~$i=1,2$. Suppose that $\lim\limits_{t\nearrow  t_i}\|u_{i}(\cdot,t)\|_{\ell^{\infty}(V)}=\infty$, $i=1,2$.
If	
 \begin{equation*}\left\{\begin{array}{ll}
u_{1t}-\Delta u_1\geq u_1^p,\;\;&(x,t)\in V\times[0, t_1),\\[1mm]
u_1(x, 0)\geq\psi(x),&x\in V,
		\end{array}\right.
	\end{equation*}
and
	\begin{equation*}\left\{\begin{array}{ll}
	u_{2t}-\Delta u_2\leq u_2^p,\;\;&(x,t)\in V\times[0, t_2),\\[1mm]
	u_2(x, 0)\leq\psi(x),~&x\in V,
		\end{array}\right.
	\end{equation*}
 then $t_1\le t_2$.
\end{lemma}

 \begin{proof} Suppose by way of contradiction that $t_1>t_2$. Then by Lemma \ref{l3.2},
	\begin{equation*}
		u_1\left(x,t\right)\geq u_2\left(x,t\right),\;\;\;(x,t)\in V\times[0, t_2).
	\end{equation*}
Since $\lim\limits_{t\nearrow t_2}\|u_2(\cdot,t)\|_{\ell^{\infty}(V)}=\infty$, it follows that $\lim\limits_{t\nearrow t_2}\|u_1(\cdot,t)\|_{\ell^{\infty}(V)}=\infty$. This is a contradiction.
\end{proof}

\begin{lemma}\label{l4.3} Assume $D_{\mu}<+\infty$. There exists a constant $\Lambda=\Lambda(p,\psi,V)$ such that when $\lambda>\Lambda$, there hold $T_\lambda <\infty$ and
  \begin{equation}\label{4.7}
\lim_{\lambda\to \infty}\lambda^{p-1}T_\lambda =\frac1{(p-1)\|\psi\|_{\ell^{\infty}(V)}^{p-1}}.
	\end{equation}
\end{lemma}

\begin{proof} {\bf Step 1}. As $\psi$ satisfies \eqref{1.7}, we can find a finite subset  $\Omega\subset V$ such that $\psi(x)\not\equiv 0$ on $\Omega$. By Lemma \ref{l4.1}, we see that if
	\begin{equation}\label{4.8}
		\lambda>\lambda_1^{\frac1{p-1}}\left(\int_{\Omega}\psi\phi {\rm d}\mu\right)^{-1},
	\end{equation}
then $T_\lambda <\infty$, where $\lambda_1$ and $\phi$ are given in Lemma \ref{l4.1}.

Now we suppose that $\lambda>\lambda_1^{\frac1{p-1}}\left(\int_{\Omega}\psi\phi {\rm d}\mu\right)^{-1}$. It is easy to check that for any given $0<T_0<\frac{\left(\lambda\|\psi\|_{\ell^{\infty}(V)}\right)^{-(p-1)}}{p-1}$, the function
\begin{equation*}
\bar{u}(t)=\left[\left(\lambda\|\psi\|_{\ell^{\infty}(V)}\right)^{-(p-1)}-(p-1)t\right]^{-\frac1{p-1}}
\end{equation*}
is an upper solution of \eqref{1.6} in $V\times[0, T_0]$. Thus, by Lemma \ref{l4.2},
 \begin{equation}\label{4.9} T_\lambda\ge\frac{\left(\lambda\|\psi\|_{\ell^{\infty}(V)}\right)^{-(p-1)}}{p-1},
\;\;\;{\rm i.e.}, \;\;\lambda^{p-1}T_\lambda \ge\frac1{(p-1)\|\psi\|_{\ell^{\infty}(V)}^{p-1}}.
\end{equation}

{\bf Step 2}. Let us recall that the function $\eta(t)=\int_{\Omega}u(\cdot,t)\phi(\cdot){\rm d}\mu$, defined by \eqref{4.2}, is continuous in $[0,T_\lambda )$ and satisfies
 \begin{equation}\label{4.10}
 	\eta'(t)\geq\eta^p(t)-\lambda_1 \eta(t),\;\; 0<t<T_\lambda;\;\;\;
 	\eta(0)=\lambda\int_{\Omega}\psi\phi {\rm d}\mu>0.
 \end{equation}
Define
 $$\eta'(0)=\int_{\Omega}(\Delta u+u^{p})(\cdot,0)\phi(\cdot){\rm d}\mu.$$
Noticing that $u\in C_V([0, T_\lambda))$ and
 $$\eta'(t)=\int_{\Omega}(\Delta u+u^{p})(\cdot,t)\phi(\cdot){\rm d}\mu,\quad 0<t<T_\lambda .$$
Using the expression of $\Delta u(x,t)$ we can see $\lim\limits_{t\to 0}\eta'(t)=\eta'(0)$.
Hence, $\eta'(t)$ is continuous in $[0,T_\lambda )$.
Let $t\to 0^+$ in the first equation of \eqref{4.10} to derive $\eta'(0)\ge\eta^p(0)-\lambda_1 \eta(0)$. Thanks to \eqref{4.8},
  \begin{equation}\label{4.12}
  	\eta(0)=\int_{\Omega}u(x,0)\phi(x){\rm d}\mu>\lambda_1^{\frac1{p-1}},
  \end{equation}
which implies $\eta^p(0)-\lambda_1 \eta(0)>0$, and $\eta'(0)>0$.
There exists $\delta\in (0,T_\lambda )$ such that $\eta(t)>0$ and $\eta'(t)>0$ for $0<t<\delta$. Define
 \bess
	T=\sup\left\{0<\delta<T_\lambda :\, \eta(t)>0, \eta'(t)>0~\text{for}~0<t<\delta\right\} .
\eess

{\bf Step 3}. We claim that $T=T_\lambda $. Suppose by way of contradiction that $T<T_\lambda $. By the definition of $T$, we see that $\eta'(t)>0$ for $0<t<T$, which implies $\eta(T)>\eta(0)>\lambda_1^{\frac1{p-1}}>0$
by \eqref{4.12}. This combines with \eqref{4.10} indicates that $\eta'(T)\geq\eta^p(T)-\lambda_1\eta(T)>0$.
This is impossible as we have assumed $T<T_\lambda $. Therefore, $T=T_\lambda $.

{\bf Step 4}. It follows from Step 3 that $\eta'(t)>0$ and $\eta(t)>0$ for $t\in[0,T_\lambda)$. This combined with \eqref{4.12} allows us to derive
$\eta(t)>\eta(0)>\lambda_1^{\frac1{p-1}}$ for $t\in (0,T_\lambda)$.
For any given $t\in (0,T_\lambda )$. According to \eqref{4.10},
 \bess
 t&\leq&\int_{\eta(0)}^{\eta\left(t\right)}\frac{{\rm d}y}{y^p-\lambda_1y}\nm\\ &=&\left.\frac1{\lambda_1(p-1)}\ln\frac{y^{p-1}-\lambda_1}{y^{p-1}}
 \right|_{\eta(0)}^{\eta\left(t\right)}\nm\\[1mm]
 &=&\frac1{\lambda_1(p-1)}\left(\ln\frac{\eta^{p-1}\left(t\right)-\lambda_1}
 {\eta^{p-1}\left(t\right)}-\ln\frac{\eta^{p-1}(0)-\lambda_1}
 {\eta^{p-1}(0)}\right)\\[1mm]
 &\leq&-\frac1{\lambda_1(p-1)}\ln\frac{\eta^{p-1}(0)-\lambda_1}{\eta^{p-1}(0)}.
 	\eess
The arbitrariness of $t$ implies
 \begin{equation*}	T_\lambda \leq-\frac1{\lambda_1(p-1)}\ln\frac{\eta^{p-1}(0)
 -\lambda_1}{\eta^{p-1}(0)}=-\frac1{\lambda_1(p-1)}\ln
 \left(1-\frac{\lambda_1}{\eta^{p-1}(0)}\right).
 \end{equation*}
Using the second equation in \eqref{4.10} and $\lim\limits_{y\to0^+}\frac{\ln(1-\lambda_1y)}y=-\lambda_1$, we have
 \bes\label{4.13}
 \limsup_{\lambda\to \infty}\lambda^{p-1}T_\lambda&\leq&\limsup_{\lambda \to \infty}\frac{\lambda^{p-1}}{\lambda_1(p-1)}
 \left[-\ln\left(1-\frac{\lambda_1}{\eta^{p-1}(0)}\right)\right]\nm\\[1mm]
&=&\lim_{\lambda\to \infty}\frac{\lambda^{p-1}}{\lambda_1(p-1)}\frac{\lambda_1}{\lambda^{p-1}\left(\int_{\Omega} \psi \phi d \mu\right)^{p-1}}\nm\\
&=&\frac1{(p-1)\left(\int_{\Omega}\psi\phi {\rm d}\mu\right)^{p-1}}.
	\ees
For the given $\tilde x \in\{x\in V:\,\psi(x)\not=0\}$. We take $\Omega=\{\tilde x\}$. Since $\int_{\Omega}\phi d\mu=1$, we have $\phi(\tilde x)\mu(\tilde x)=1$. From this and \eqref{4.13}, we see $\limsup_{\lambda\to\infty}\lambda^{p-1}T_\lambda\leq\frac1{(p-1)\psi^{p-1}(\tilde x)}$.
The arbitrariness of $\tilde x$ implies
  \bess
\limsup_{\lambda\to\infty}\lambda^{p-1}T_\lambda\leq\frac1{(p-1)\|\psi\|_{\ell^{\infty}(V)}^{p-1}}.
\eess
This combined with \eqref{4.9} allows us to derive \eqref{4.7}.
\end{proof}

Theorem \ref{t2.1} follows from Lemma \ref{l4.3} immediately.

\subsection{Proof of Theorem \ref{t2.2}}{\setlength\arraycolsep{2pt}

To prove Theorem \ref{t2.2}, we first prove some lemmas.

\begin{lemma}\label{l4.4}
Suppose $D_{\mu}, D_{\omega}<\infty$ and $\inf\limits_{V}\mu>0$.  Assume that $G=(V,E)$ satisfies the hypotheses of Theorem \ref{t2.2}. Fix $\tau>0$. If there exists $\bar x\in V$ such that $\liminf_{d\left(\bar x,x\right)\to\infty}\psi(x)>0$,  then
\begin{equation}\label{4.14}
\inf_{x\in V}\sum_{y\in V}P(\tau,x,y)\psi(y)\mu(y)>0.
\end{equation}
\end{lemma}

\begin{proof}[Proof] Letting $\psi_{\infty}=\liminf\limits_{d\left(\bar x,x\right)\to\infty}\psi(x)$. Without loss of generality, we assume that there exist $0<\varepsilon<\psi_{\infty}$ and a positive integer $k$ such that
 \begin{equation}\label{4.15}
\psi(x)= \begin{cases}0, &d(x,\bar x)\leq k,\\[0.1mm]
\psi_{\infty}-\varepsilon, &d\left(x, \bar x\right)\geq k+1.\end{cases}
 \end{equation}
Since $V$ is locally finite, we have that
 \bess
\#\{x\in V: d(\bar x,x)<k+1\}<\infty,\;\;\;\#\left\{y\in V:\, k+1\le d(\bar x,y)\leq
  k+2\right\}<\infty.
 \eess		
Fix $\tau>0$. Thanks to $P(\tau,x,y), \psi(y)$ and $\mu(y)>0$ for $x,y\in V$, by \eqref{4.15}, we see  that
 \bes\label{4.16}
	 \inf\limits_{\substack{x\in V\\ d(x,\bar x)<k+1}}\;\sum_{y\in V}P(\tau,x,y)\psi(y)\mu(y)&\ge& \inf\limits_{\substack{x\in V\\ d(x,\bar x)<k+1}}\sum_{\substack{y\in V\\ k+1\leq d(\bar x,y)\leq k+2}}\!\!\!\!\!\!\!\!P(\tau,x,y)\psi(y)\mu(y)\nm\\[1mm]
	&=&\inf\limits_{\substack{x\in V\\d(x,\bar x)<k+1}}\sum_{\substack{y\in V\\ k+1\leq d(\bar x,y)\leq k+2}}\!\!\!\!\!\!\!\! P(\tau,x,y)\left(\psi_{\infty}-\varepsilon\right)\mu(y)\nm\\
 &=&\min_{\substack{x\in V\\ d(x,\bar x)<k+1}}\sum_{\substack{y\in V\\ k+1\leq d(\bar x,y)\leq k+2}}\!\!\!\!\!\!\!\!P(\tau,x, y)\left(\psi_{\infty}-\varepsilon\right)\mu(y)\nm\\
 &>0&
\ees
and
\begin{equation}\label{4.17}
\inf_{\substack{x\in V\\ {d(x,\bar x)\geq k+1}}}\;\sum_{y\in V}P(\tau,x,y)\psi(y)\mu(y)\geq\inf_{\substack{x\in V\\{d(x,\bar x)\geq k+1}}}\!\!P(\tau,x,x)\psi(x)\mu(x).
\end{equation}
Since the graph $G$ satisfies $CDE'(n,0)$, $D_{\mu}, D_{\omega}<\infty$, $\psi_{\infty}-\varepsilon>0$ and $\inf\limits_{V}\mu >0$. Using \eqref{4.15}, Proposition \ref{p2.3} and the condition {\bf(VG)} in sequence, we have
 \bess
 \inf_{\substack{x\in V\\ {d(x,\bar x)\geq k+1}}}\!\!P(\tau,x,x)\psi(x)\mu(x)&\geq&\inf_{\substack{x\in V\\{d(x,\bar x)\geq k+1}}}\frac{C(n)}{\mathcal{V}(x,\sqrt{\tau})}(\psi_{\infty}-\varepsilon)\inf_V\mu\nm\\
 &\geq&\inf_{\substack{x\in V\\ {d(x,\bar x)\geq k+1}}}\frac{C(n)}{c_0 \tau^{\frac{m}2}}(\psi_{\infty}-\varepsilon)\inf_V\mu\nm\\[1mm]
 &>&0.
 \eess
This combines with \eqref{4.16} and \eqref{4.17} to deduce \eqref{4.14}. The proof is complete.
\end{proof}

\begin{lemma}\label{l4.5}
 Assume that the hypotheses in Lemma \ref{l4.4} hold. Then for any $\tau>0$,
 there holds
 \bes
 \sum_{y\in V}P(t,x,y)\psi(y)\mu(y)\ge\sigma_0(\tau),\;\;~x\in V,~t>\tau,
 \label{4.18}\ees
where
\begin{equation*}
\sigma_0(\tau)=\inf_{x\in V}\sum_{y\in V}P(\tau,x,y)\psi(y)\mu(y)>0.
	\end{equation*}
\end{lemma}

\begin{proof} Clearly, $\sigma_0(\tau)>0$ by Lemma \ref{l4.4}. In view of Proposition \ref{p2.2} (iv) and (v), it follows that
  \bess
	\sum_{y\in V}\sum_{z\in V}P(t-\tau,x,z)P(\tau,z,y)\psi(y)\mu(z)\mu(y)
	&=&\sum_{y\in V}\left(\sum_{z\in V}P(t-\tau,x,z)P(\tau,z,y)\mu(z)\right) \psi(y)\mu(y)\\
	&=&\sum_{y\in V}P(t,x,y)\psi(y)\mu(y)\\
	&\leq&\|\psi\|_{\ell^{\infty}(V)}\;~~\text{for}~x\in V,~t>\tau.
	\eess
Using this and Proposition \ref{p2.2} (v), we have 	
 \bess
\sum_{y\in V}P(t,x,y)\psi(y)\mu(y)&=&\sum_{y\in V}\left(\sum_{z\in V} P(t-\tau,x,z)P(\tau,z,y)\mu(z)\right)\psi(y)\mu(y)\\
&=&\sum_{z\in V}\sum_{y \in V}P(t-\tau,x,z) P(\tau,z,y)\psi(y)\mu(y)\mu(z).
  \eess		
Since $\sum\limits_{y\in V}P(\tau,z,y)\psi(y)\mu(y)\ge \sigma_0(\tau)$ for $z\in V$ and $P(t-\tau,x,z)\mu(z)>0$ for $t>\tau$ and $x,z\in V$. It is easy to see that
\bess
	\sum_{z\in V}\sum_{y \in V}P(t-\tau,x,z)P(\tau,z,y)\psi(y)\mu(y)\mu(z)
	&=&\sum_{z\in V}P(t-\tau,x,z)\mu(z)\sum_{y \in V}P(\tau,z,y)\psi(y)\mu(y)\\	
  &\geq&\sum_{z\in V}P(t-\tau,x,z) \sigma_0(\tau)\mu(z),\;\;~x\in V,~t>\tau.\eess
By virtue of $D_{\mu}<\infty$ and Lemma \ref{l3.3}, it follows that
	\bess
\sum_{z\in V}P(t-\tau,x,z) \sigma_0(\tau)\mu(z)=\sigma_0(\tau),\;\;~x\in V,~t>\tau.\eess
Based on the above discussions, it is easy to see that \eqref{4.18} holds true.
\end{proof}
	
\begin{lemma}\label{l4.6} Suppose $D_{\mu}<\infty$. If $u$ solves \eqref{1.6} in  $[0,T_\lambda )$, then for any vertex $\bar x \in V$,
	\begin{equation}
		\left(\lambda\sum_{y\in V}P(t,\bar x,y)\psi(y)\mu(y)\right)^{1-p}\geq(p-1)t,\quad 0<t<T_\lambda .\label{4.19}
	\end{equation}
\end{lemma}

\begin{proof}\; For any $0<t<T_\lambda $ and any $\bar{t}\in[0,t)$. According to Theorem \ref{t3.1},
	\bes\label{4.20}
u(x,\bar{t})=\lambda\sum_{y\in V}P(\bar{t},x,y)\psi(y)\mu(y)+\int_0^{\bar{t}}\sum_{y\in V}P(\bar{t}-s,x,y)u^p(y,s)\mu(y){\rm d}s.
	\ees

We next show that
\bes\label{4.21}
\sum_{x\in V}P(t-\bar{t},\bar x,x)u(x,\bar{t})\mu(x)&=&\sum_{x\in V}P(t-\bar{t},\bar x,x)\sum_{y \in V}P(\bar{t},x,y)\lambda\psi(y)\mu(y)\mu(x)\nm\\
	&&+\sum_{x\in V}P(t-\bar{t},\bar x,x)\!\int_0^{\bar{t}}\sum_{y\in V}P(\bar{t}-s,x,y)u^p(y,s)\mu(y){\rm d}s\mu(x).\qquad
\ees
Clearly, there exists $C=C(t)$ such that
  \bes\label{4.22}|u^{p}(y,s)|\le C,\;\;\;\forall\, y\in V,~s\in[0,\bar{t}].\ees
Using this and Proposition \ref{p2.2} (iv) we have
 \bes\label{4.23}	
\sum_{x,y\in V}\!P(t-\bar{t},\bar x,x)P(\bar{t}-s,x,y)u^p(y,s)\mu(y)\mu(x)
&\leq&C\sum_{x,y\in V}\!P(t-\bar{t},\bar x,x)P(\bar{t}-s,x,y)\mu(y)\mu(x)\nm\\
&\leq&C\sum_{x\in V}\!P(t-\bar{t},\bar x,x)\mu(x)\sum_{y\in V}\!P(\bar{t}-s,x,y) \mu(y)\nm\\
&\leq&C.
 \ees
From \eqref{4.22} and Proposition \ref{p2.4}, we know that $\sum_{y\in V}P(t-\bar{t},\bar x,x)P(\bar{t}-s,x,y)u^{p}(y,s)\mu(y)\mu(x)$ converges uniformly w.r.t. $s\in[0,\bar{t})$. As $P(t-\bar{t},\bar x,x)P(\bar{t}-s,x,y)u^{p}(y,s)\mu(y)\mu(x)$ is continuous w.r.t. $s\in[0,\bar{t})$, we see that
\begin{equation}\label{4.24}
	\sum\limits_{y\in V} P(t-\bar{t},\bar x,x)P(\bar{t}-s,x,y)u^{p}(y,s)\mu(y)\mu(x)\text{ is continuous w.r.t. }s\in[0,\bar{t}).
\end{equation}
Thanks to \eqref{4.22} and Lemma \ref{l3.3}, it can be deduced that
  $\big| \sum_{y\in V}P(\bar{t}-s,x,y)u^{p}(y,s)\mu(y)\big| \le C$.
Hence, by Proposition \ref{p2.4}, $\sum_{x\in V}P(t-\bar{t},\bar x,x)\sum_{y\in V}P(\bar{t}-s,x,y)u^{p}(y,s)\mu(y)\mu(x)$ converges uniformly w.r.t. $s\in[0,\bar{t})$. By virtue of \eqref{4.24} and \eqref{4.23}, it follows that
\bes\label{4.25}
	&&\sum_{x\in V}P(t-\bar{t},\bar x,x)\int_0^{\bar{t}}\sum_{y\in V} P(\bar{t}-s,x,y)u^p(y,s)\mu(y) \mu(x){\rm d}s\nm\\
	&=&\int_0^{\bar{t}}\sum_{x\in V}\sum_{y\in V} P(t-\bar{t},\bar x,x)P(\bar{t}-s,x,y)u^p(y,s)\mu(y)\mu(x){\rm d}s\nm\\
	&=&\int_0^{\bar{t}}\sum_{y\in V}\sum_{x\in V} P(t-\bar{t},\bar x,x)P(\bar{t}-s,x,y)u^p(y,s)\mu(y)\mu(x){\rm d}s\nm\\
	&=:&I.
	\ees
Making use of Proposition \ref{p2.2} (iv), (v) and \eqref{4.22} one has
 \bess
I&=&\int_0^{\bar{t}}\sum_{y\in V}u^p(y,s)\mu(y)\sum_{x\in V} P(t-\bar{t},\bar x,x)P(\bar{t}-s,x,y)\mu(x){\rm d}s\nm\\
&=&\int_0^{\bar{t}}\sum_{y \in V}P(t-s,\bar x,y)u^p(y,s)\mu(y){\rm d}s\nm\\
 &\le& C\bar{t}.
 \eess
This combined with \eqref{4.25} allows us to derive
\bes\label{4.26}
\sum_{x\in V}P(t-\bar{t},\bar x,x)\int_0^{\bar{t}}\sum_{y\in V} P(\bar{t}-s,x,y)u^p(y,s)\mu(y) \mu(x){\rm d}s&=&\int_0^{\bar{t}}\sum_{y \in V}P(t-s,\bar x,y)u^p(y,s)\mu(y){\rm d}s\nm\\
&\le&C\bar{t}.
\ees
By similar arguments as above, we can obtain
 \bes\label{4.27}
 \lambda\sum_{x\in V}P(t-\bar{t},\bar x,x)\sum_{y\in V}P(\bar{t}, x,y)\psi(y)\mu(y)\mu(x)&=&\lambda\sum_{y\in V}P(t,\bar x,y)\psi(y)\mu(y)\nm\\
&\le&\lambda\|\psi\|_{\ell^{\infty}(V)}.
\ees
On the basis of \eqref{4.20}, \eqref{4.26} and \eqref{4.27}, it is easy to derive \eqref{4.21}.
By comprehensively utilizing \eqref{4.21}, \eqref{4.25}, \eqref{4.26} and \eqref{4.27}, it can be concluded that
 \bes
\sum_{x\in V}\!P(t\!-\!\bar{t},\bar x,x)u(x,\bar{t})\mu(x)
&=&\lambda\sum_{y\in V}\!P(t,\bar x,y)\psi(y)\mu(y)\!+\!\int_0^{\bar{t}}\sum_{y\in V}\!P(t\!-\!s,\bar x,y)u^p(y,s)\mu(y){\rm d}s.\qquad\label{4.28}
\ees

Due to $D_{\mu}<\infty$, in view of Lemma \ref{l3.3} we have $\sum_{y\in V}P(t-s,\bar x,y)\mu(y)=1$ for $0<s<t$. It follows by use of Jensen's inequality that
 \bess
	\sum_{y\in V}P(t-s,\bar x,y)u^p(y,s)\mu(y)\ge \left(\sum_{y\in V} P(t-s,\bar x,y)u(y,s)\mu(y)\right)^p.
 \eess
This combined with \eqref{4.28} allows us to deduce
 \bess
\sum_{x\in V}\!P(t-\bar{t},\bar x,x)u(x,\bar{t})\mu(x)
&\geq&\lambda\sum_{y\in V}\!P(t,\bar x,y)\psi(y)\mu(y)+\int_0^{\bar{t}}\left(\sum_{y\in V}\!P(t-s,\bar x,y)u(y,s)\mu(y)\right)^p {\rm d}s\nm\\[1mm]
&=:& G(\bar{t}).
	\eess
Clearly,
  \begin{equation}\label{4.29}
G(\bar{t})\ge G(0)=\lambda\sum_{y\in V}P(t,\bar x,y)\psi(y)\mu(y)
\end{equation}
and
 \begin{equation}\label{4.30}
\frac{{\rm d}G}{{\rm d}\bar t}=\left(\sum_{y\in V}P(t-\bar{t},\bar x,y)u(y,\bar{t})\mu(y)\right)^p\geq G^p(\bar{t}).
 \end{equation}
Due to $\psi(x)\ge,\not\equiv 0$ on $V$, there exists $x_1\in V$ such that $\psi(x_1)>0$. Thus, by \eqref{4.29}, $G(\bar{t})\ge\lambda P(t,\bar x,x_1)\psi(x_1)\mu(x_1)>0$. Therefore
$\frac1{1-p}\left[G^{1-p}(\bar{t})-G^{1-p}(0)\right]\geq \bar{t}$ by \eqref{4.30}. As $p>1$ and $G(\bar{t})>0$, it follows that $G^{1-p}(0)\geq(p-1)\bar{t}$,
i.e.,
\begin{equation*}
	\left(\lambda\sum_{y\in V}P(t,\bar x,y)\psi(y)\mu(y)\right)^{1-p}\geq(p-1)\bar{t}.
\end{equation*}
Letting $\bar{t} \to t$ we conclude that \eqref{4.19} holds.
\end{proof}

\begin{proof}[Proof of Theorem \ref{t2.2}]	For any $\lambda>0$, choose $0<\tau_{\lambda}<T_\lambda $. Then, by \eqref{4.19},
 \begin{equation}\label{4.31}
 \left(\lambda\sum_{y\in V}P(t,\bar x,y)\psi(y)\mu(y)\right)^{1-p}\geq(p-1)t,~\tau_\lambda<t<T_\lambda .
	\end{equation}
In view of Lemma \ref{l4.5}, $\sum_{y\in V}P(t,\bar x,y)\psi(y)\mu(y)\ge \sigma_0(\tau_\lambda)$ for $\tau_\lambda<t<T_\lambda$. This combined with \eqref{4.31} allows us to derive $\sigma_0^{1-p}\left(\tau_\lambda\right)\lambda^{1-p}\geq (p-1)t$ for $\tau_\lambda<t<T_\lambda$. This implies $T_\lambda <\infty$ for any $\lambda>0$. The proof is complete.
\end{proof}

\subsection{Proof of Theorem \ref{t2.3}}

\begin{lemma}\label{l4.7}
	Suppose that there exist $\tilde x\in V$ and $\psi_{\infty}>0$ such that $		 \lim\limits_{d\left(\tilde x,x\right)\to\infty}\psi(x)=\psi_{\infty}$,
and that $G$ satisfies the condition {\bf(EC)}. Then
 \bess
 \lambda^{p-1}T_\lambda\leq \frac1{(p-1)\psi_\infty^{p-1}}.
\eess
\end{lemma}

\begin{proof} By our assumption, for any given $\varepsilon\in(0,\min\{1,\psi_{\infty}\})$, there exists $\delta>0$ such that 	
\begin{equation}\label{4.32}
\psi(x)>\psi_{\infty}-\varepsilon\quad\text{ for }x\in\left\{x\in V: d\left(\tilde x, x\right)>\delta\right\}.
\end{equation}
For any given $\lambda>0$, choosing
\bess
 \varepsilon'=\varepsilon\lambda^{p-1}(\psi_{\infty}-\varepsilon)^{p-1},\;\;\;\text{and} \;\;\delta'>\delta.
\eess
Since $G$ satisfies the condition {\bf(EC)}, there exists $\Omega=\Omega(\varepsilon',\delta')\subset V$ such that
$\lambda_1(\Omega)<\varepsilon'$ and $d\left(x,\tilde x\right)>\delta'$
for all $x\in\Omega$. Then we have $\psi(x)>\psi_{\infty}-\varepsilon$ in $\Omega$ by \eqref{4.32}. Let $\phi(x)>0$ be the normalized eigenfunction corresponding to $\lambda_1(\Omega)$, i.e., $\int_{\Omega}\phi{\rm d}\mu=1$. It follows that
 \bess
\varepsilon\lambda^{p-1}\left(\int_{\Omega}\psi(x)\phi(x){\rm d}\mu\right)^{p-1}>\varepsilon \lambda^{p-1}\left(\psi_{\infty}-\varepsilon\right)^{p-1}>\lambda_1(\Omega).
 \eess
Define $\eta(t)$ as in \eqref{4.2}.
Then $\varepsilon\eta^{p-1}(0)>\lambda_1(\Omega)$, and so
$\eta^p(0)-\lambda_1(\Omega)\eta(0)>(1-\varepsilon)\eta^p(0)>0$.
By a similar discussion as in the proof of Lemma \ref{l4.3}, we deduce that
\begin{equation}\label{4.33}
 \eta'(t)\ge\eta^{p}(t)-\lambda_1(\Omega)\eta(t),~\eta(t)>0,\text{ and }\eta'(t)>0,~t\in(0, T_\lambda ).
\end{equation}
This combines with $\varepsilon\eta^{p-1}(0)>\lambda_1(\Omega)$ indicates $\varepsilon\eta^{p-1}(t)>\varepsilon\eta^{p-1}(0)>\lambda_1(\Omega)$, and hence $\eta^{p}(t)-\lambda_1(\Omega)\eta(t)>(1-\varepsilon)\eta^p(t)>0$ for $t\in(0, T_\lambda)$. From this and \eqref{4.33} we have $\eta'(t)>(1-\varepsilon)\eta^p(t)$ for $t\in(0, T_\lambda)$.
It follows that, for any given $t_1\in(0,T_\lambda)$,
 \bess
t_1&\leq&\frac1{1-\varepsilon}\int_{\eta(0)}^{\eta(t_1)}\frac{{\rm d}\eta}{\eta^p}\nm\\
&=&\frac 1{(1-\varepsilon)(p-1)}\left[\eta^{1-p}(0)-\eta^{1-p}(t_1)\right] \nm\\
&<&\frac 1{(1-\varepsilon)(p-1)}\eta^{1-p}(0)\nm\\
&=&\frac1{\lambda^{p-1}(1-\varepsilon)(p-1)\left[\int_{\Omega}\psi(\cdot)\phi(\cdot){\rm d}\mu\right]^{p-1}}.
 \eess
The arbitrariness of $t_1$ implies
 \bess
T_\lambda \le\frac1{\lambda^{p-1}(1-\varepsilon)(p-1)\left(\int_{\Omega}\psi(\cdot)\phi(\cdot){\rm d}\mu\right)^{p-1}}.
 \eess
Making use of $\int_{\Omega}\phi {\rm d}\mu=1$ and $\psi(x)>\psi_{\infty}-\varepsilon$ in $\Omega$  we have
\bess
\lambda^{p-1}T_\lambda \leq\frac1{(1-\varepsilon)(p-1)\left(\int_{\Omega}\psi(\cdot)\phi(\cdot){\rm d}\mu\right)^{p-1}}
\le\frac1{(1-\varepsilon)(p-1)\left(\psi_{\infty}-\varepsilon\right)^{p-1}}.
\eess
The arbitrariness of $\varepsilon$ completes the proof.
\end{proof}

\begin{lemma}\label{l4.8} Suppose $D_{\mu}<+\infty$. Assume the hypotheses of Lemma \ref{l4.7} hold. Suppose that $\psi(x)\le \psi_{\infty}, x\in V$, where $\psi_{\infty}$ is a positive constant. Then
\bes
\liminf\limits_{\lambda\to0}\lambda^{p-1}T_\lambda\geq
\frac1{(p-1)\psi_{\infty}^{p-1}}.\label{4.34}
\ees
\end{lemma}

\begin{proof} Set $T_*=\frac1{(p-1)\lambda^{p-1}\psi_{\infty}^{p-1}}$.
It is easy to see that 	\begin{equation*}
\bar{v}(t)=\big[\left(\lambda\psi_{\infty}\right)^{-(p-1)}-(p-1)t\big]^{-\frac1{p-1}}
\end{equation*}
satisfies \begin{equation*}
	\left\{\begin{array}{lll}
		{\bar{v}}'(t)=\bar{v}^p(t),~&0<t<T_*,\\[1.5mm]
		{\bar{v}}(0)=\lambda\psi_{\infty}\geq \lambda\psi(x),~&x\in V.
	\end{array}\right.
 \end{equation*}
Noticing $\lim\limits_{t\nearrow T_*}\bar{v}(t)=\infty$, and $\lim\limits_{t\nearrow T[\lambda \psi]}\|u(\cdot,t)\|_{\ell^{\infty}(V)}=\infty$.  By Lemma \ref{l4.2}, $T_\lambda \geq\frac1{(p-1)\lambda^{p-1}\psi_{\infty}^{p-1}}$.	This implies \eqref{4.34}.
\end{proof}

Theorem \ref{t2.3} ($1$) follows from Lemmas \ref{l4.7} and \ref{l4.8} immediately.

\begin{proof}[Proof of Theorem \ref{t2.3} $(2)$]
Let $T_0=\frac1{(p-1)\lambda^{p-1}\|\psi\|_{\ell^{\infty}(V)}^{p-1}}$ and
  \begin{equation*}
	\begin{aligned}	 S_1(t)=\big[\left(\lambda\|\psi\|_{\ell^{\infty}(V)}\right)^{-(p-1)}-(p-1)t\big]^{-\frac1{p-1}}.
	\end{aligned}
\end{equation*}
Clearly, $S_1(t)$ satisfies
  \bess
	\left\{\begin{array}{ll}
	{S_1}'(t)=S_1^p(t),~&0<t<T_0,\\[1.5mm]
	{S_1}(0)=\lambda\|\psi\|_{\ell^{\infty}(V)}\geq\lambda\psi(x),~&x\in V,
	\end{array}\right.\eess
Similar to the above, $T_0\leq T_\lambda$, which implies  $\frac1{(p-1)\|\psi\|_{\ell^{\infty}(V)}^{p-1}}\le\liminf_{\lambda\to 0}\lambda^{p-1}T_\lambda$. By Lemma \ref{l4.7}, $\limsup_{\lambda\to 0} \lambda^{p-1}T_\lambda \le\frac1{(p-1)\psi_{\infty}^{p-1}}$.
Therefore,
 \begin{equation*}
	\frac1{(p-1)\|\psi\|_{\ell^{\infty}(V)}^{p-1}}\le\liminf_{\lambda\to 0}\lambda^{p-1} T_\lambda \le \limsup_{\lambda \to 0} \lambda^{p-1} T_\lambda \le\frac1{(p-1) \psi_{\infty}^{p-1}}.
\end{equation*}
The proof is complete.
\end{proof}

\section{Proof of Theorem \ref{t2.4}}

In this section, we suppose that $G=(V,E)$ is a finite connected graph.

\begin{lemma}
 Assume that $u(x,t)$ is the solution of \eqref{1.6}.
Then there exists constant $\Lambda_1=\Lambda_1(p,\psi,V)$ such that $T_\lambda <\infty$ when $\lambda>\Lambda_1$.	
\end{lemma}

\begin{proof} Select a vertex $\tilde x\in V$. We add a new vertex $z$ and a new edge $\tilde x\sim z$, and then define a new graph $G_1=G_1(V_1,E_1)$, where $V_1=V \cup\{z\}, E_1=E\cup\left\{\tilde xz\right\}$. We may extend $\omega$ to the set $E_1$ by setting $\tilde{\omega}_{yx}\arrowvert_{E}=\omega$, $\tilde{\omega}_{z\tilde{x}}=1$, for convenience, $\tilde{\omega}$ is still denoted by $\omega$, and extend $u(\cdot,t)$ to $\{z\}$ by letting $u(z,t)=0$. Let $\lambda_1$ be the smallest eigenvalue of the eigenvalue problem
	\begin{equation}\label{5.1}
\left\{\begin{array}{ll}
-\Delta_{V_1}\phi(x)=-\dd\frac{1}{\mu(x)}\sum\limits_{V_1\ni y\sim x}[\phi(y)-\phi(x)]\omega_{y x}=\lambda\phi(x),\;\;&x\in V,\\[3mm]
\phi(x)=0,&x\in\partial V_1=\left\{z\right\}
 \end{array}\right.	\end{equation}
and $\phi$ be the eigenfunction corresponding to $\lambda_1$ satisfying $\int_{V}\phi {\rm d}\mu=1$ and  $\phi>0$ on $V$.
Define $\eta(t)$ as in \eqref{4.2} with $\Omega=V$. Then
 $$\eta'(t)=\int_V \left[\Delta u(\cdot,t)+u^p(\cdot,t)\right]\phi(\cdot){\rm d}\mu.$$

We next show that
\begin{equation}\label{5.2}
\Delta u(x,t)\ge\Delta_{V_1}u(x,t)\;\;~\text{for}~(x,t)\in V\times(0, T_\lambda).
\end{equation}
For any $x\in V$. If $x=\tilde x$, recalling that $u(z,t)=0$, and $u(x,t)\ge 0$ for $(x,t)\in V\times(0, T_\lambda)$, we deduce that
\bess
\Delta_{V_1}u(\tilde x,t)&=&\frac{1}{\mu(\tilde x)}\sum_{y\in V_1,\,y\sim \tilde x}[u(y,t)-u(\tilde x,t)]\omega_{y\tilde x}\\
&=&\frac{1}{\mu(\tilde x)}\Bigg([u(z,t)-u(\tilde x,t)]\omega_{z\tilde x}
+\sum_{y\in V,\,y_{\sim}\tilde x}[u(y,t)-u(\tilde x,t)]\omega_{y \tilde x}\Bigg),\\
&\leq&\frac{1}{\mu(\tilde x)}\sum_{y\in V,\,y_{\sim}\tilde x}[u(y,t)-u(\tilde x,t)]\omega_{y \tilde x},\\
&=&\Delta u(\tilde x,t).
\eess
If $x\not=\tilde x$, it is easy to check that $\Delta u(x,t)=\Delta_{V_1}u(x,t)$. Thus, \eqref{5.2} holds.

Since $u(z,t)=0$, $\phi(z)=0$ and $\phi>0$ on $V$, integrating by parts and using \eqref{5.1} we have
\bess
\eta'(t)&=&\int_{V}\left(\Delta u(\cdot,t)+u^p(\cdot,t)\right)\phi(\cdot){\rm d}\mu\\
&\geq&\int_{V_{}}\left(\Delta_{V_1}u(\cdot,t)+u^p(\cdot,t)\right)\phi(\cdot)d \mu\\
&=&\int_{V_1}\left(\Delta_{V_1}u(\cdot,t)+u^p(\cdot,t)\right)\phi(\cdot){\rm d}\mu\\
&=&\int_{V_1}[u(\cdot,t)\Delta_{V_1}\phi(\cdot)+u^p(\cdot,t)]\phi(\cdot){\rm d}\mu\\
&=&-\lambda_1\int_{V_1}u(\cdot,t)\phi(\cdot){\rm d}\mu+\int_{V_1} u^p(\cdot,t)\phi(\cdot){\rm d}\mu\\
&=&-\lambda_1\int_Vu(\cdot,t)\phi(\cdot){\rm d}\mu+\int_V u^p(\cdot,t)\phi(\cdot){\rm d}\mu.
	\eess
Since $\int_{V}\phi {\rm d}\mu=1$, by Jensen's inequality, it follows that
\bess
\eta'(t)\geq-\lambda_1\int_Vu(\cdot,t)\phi(\cdot){\rm d}\mu+\left(\int_V u(\cdot,t)\phi(\cdot){\rm d}\mu\right)^p
\geq\eta^p(t)-\lambda_1\eta(t),\;\;0<t<T_\lambda.
\eess
Thus, by similar arguments as in the proof of Lemma \ref{l4.1}, we deduce that if $\lambda>\lambda_1^\frac1{{p-1}}\left(\int_V \psi\phi{\rm d}\mu\right)^{-1}=\Lambda_1$, then $T_\lambda <\infty$.
\end{proof}

By a similar discussion as in the proof of Lemma \ref{l4.2} (using Lemma 2.3 in \cite{Tian} instead of Lemma \ref{l3.2}), we obtain the following result.

\begin{lemma}\label{l5.2}Let $u_i\in C_V([0,t_i))\cap C^1_V((0,t_i))$, and $\lim\limits_{t\nearrow T_i}\|u_{i}(\cdot,t)\|_{\ell^{\infty}(V)}=\infty,~i=1,2$.
If	\begin{equation*}
		\left\{\begin{array}{ll}
\partial_tu_1-\Delta u_1\geq u_1^p,~&(x,t)\in V\times[0, t_1),\\[1.5mm]
u_1(x,0)\geq\psi(x),~&x \in V,
		\end{array}\right.
	\end{equation*}
and
\begin{equation*}
	\left\{\begin{array}{ll}
\partial_tu_2-\Delta u_2\leq u_2^p,~&(x,t)\in V\times\left[0, t_2\right),\\[1.5mm]
		u_2(x,0)\leq \psi(x),~&x\in V,
	\end{array}\right.
\end{equation*}
then $t_1\le t_2$.
\end{lemma}

\begin{lemma}\label{l5.3} The following holds:
	\begin{equation}\label{5.3}
\lim_{\lambda\to \infty}\lambda^{p-1}T_\lambda =\frac1{(p-1)\big(\max_V \psi\big)^{p-1}}.
	\end{equation}
\end{lemma}

\begin{proof} Clearly, there exists $\tilde x\in V$ such that $\psi(\tilde x)=\max\limits_{V}\psi=\psi_M$. Let $\Omega=\{\tilde x\}$, and $\lambda_1$ be the smallest eigenvalue to the eigenvalue problem \eqref{2.1} and $\phi(x)$ be the normalized eigenfunction corresponding to $\lambda_1$. Let $\eta(t)$ be defined by  \eqref{4.2}. Similar to the argument as in the proof of Lemma \ref{l4.3}, we deduce that
\begin{equation*}
	\eta'(t)\geq-\lambda_1 \eta(t)+\eta^p(t),\;\;\; t\in[0, T_\lambda),
\end{equation*}
and that
\begin{equation}
\limsup_{\lambda\to \infty}\lambda^{p-1} T_\lambda \leq\frac1{(p-1)\left(\int_{\Omega}\psi\phi {\rm d}\mu\right)^{p-1}}=\frac1{(p-1)\psi^{p-1}\left(\tilde x\right)}=\frac1{(p-1)\psi_M^{p-1}}.\label{x5.2}
\end{equation}
Let $\bar{w}(t)=\big[\big(\lambda\psi_{M}\big)^{-(p-1)}-(p-1)t\big]^{-\frac1{p-1}}$ and $T_0=\frac1{(p-1)\lambda^{p-1}\psi_M^{p-1}}$. Then $\lim\limits_{t\nearrow T_0}\bar{w}(t)=\infty$ and
 \bess
\bar{w}'(t)=\bar{w}^{p}(t), \;\;t\in [0,T_0);\;\;\;\;
\bar{w}(0)=\lambda\psi_{M}\ge\lambda\psi(x)=u(x,0),~~x\in V.
 \eess
Recall that $\lim\limits_{t\nearrow T_\lambda}\|u(\cdot,t)\|_{\ell^{\infty}(V)}=\infty$. By Lemma \ref{l5.2}, $T_0\leq T_\lambda$, and hence $\frac1{(p-1)\psi_M^{p-1}}\leq\dd\liminf_{\lambda\to \infty}\lambda^{p-1}T_\lambda$. This combines with \eqref{x5.2} gives \eqref{5.3}.
\end{proof}

\begin{lemma}\label{l5.4}
	If $\min\limits_{V}\psi(x)>0$, then $T_\lambda <\infty$ for every $\lambda>0$.
\end{lemma}

\begin{proof}Since $G$ is finite, we know that
 \begin{equation*}
		D_{\mu}=\sup_{x\in V}\frac 1{\mu(x)}\sum\limits_{y\in V:y\sim x}\omega_{xy}=\max_{x\in V}\frac 1{\mu(x)}\sum\limits_{y\in V:y\sim x}\omega_{xy}<\infty.
	\end{equation*}
Hence, by Lemma \ref{l3.3}, for any $\tilde x \in V$,
  \begin{equation}\label{5.6}
	\sum_{y\in V}\mu(y)P(t,\tilde x,y)=1,\quad 0<t<T_\lambda .
 \end{equation}
Similar to the arguments as in the proof of Lemma \ref{l4.6} we can get
	\begin{equation}\label{5.5.}
		\left(\sum_{y\in V}P(t,\tilde x,y)\lambda\psi(y)\mu(y)\right)^{1-p}\geq(p-1)t,\quad 0<t<T_\lambda.
	\end{equation}
Since $\psi(x)\ge \min\limits_{V}\psi(x)>0$ for $x\in V$, it follows from \eqref{5.5.} and \eqref{5.6} that $(p-1)t\le(\lambda\min_{V}\psi)^{1-p}$ for $0<t<T_\lambda$. This implies $T_\lambda<\infty$.
\end{proof}

\begin{lemma}\label{l5.5}
	Suppose $\min\limits_{V}\psi(x)>0$. Then
	\begin{equation}\label{5.5}
\frac1{(p-1)\big(\max_{V}\psi\big)^{p-1}}\leq\liminf_{\lambda\to 0}\lambda^{p-1}T_\lambda \leq\limsup_{\lambda\to 0}\lambda^{p-1}T_\lambda \leq\frac1{(p-1){\big(\min_{V}\psi}\big)^{p-1}}.
	\end{equation}
\end{lemma}
\begin{proof}
It is easily seen that functions
 $$		 S_1(t)=\Big(\big(\lambda\max_{V}\psi\big)^{-(p-1)}-(p-1)t\Big)^{-\frac1{p-1}},\;\;\;
 S_2(t)=\Big(\big(\lambda\min_{V}\psi\big)^{-(p-1)}-(p-1)t\Big)^{-\frac1{p-1}}
	$$
satisfy
	\bess
	&&S_1'(t)=S_1^p(t),\quad t\in\left[0,t_1\right),\\
&&S_2'(t)=S_2^p(t),\quad t \in\left[0,t_2\right),\\
	&&S_1(0)\geq \lambda \psi(x)=u(x,0)\geq S_2(0),
 \eess
and $\lim_{t\nearrow t_1}S_1(t)=\lim_{t\nearrow t_2}S_2(t)=\infty$, where
$$t_1=\frac1{(p-1)\lambda^{p-1}{(\max_{V}\psi(x)})^{p-1}},\;\;\; t_2=\frac1{(p-1)\lambda^{p-1}{(\min_{V}\psi(x)})^{p-1}}.$$
Recalling that $\lim\limits_{t\nearrow T_\lambda}\|u(\cdot,t)\|_{\ell^{\infty}(V)} =\infty$. We can apply Lemma \ref{l5.2} to  $S_1(t), u(x,t)$ and $S_2(t)$ to deduce that
$t_1\leq T_\lambda\leq t_2$. Thus \eqref{5.5} holds.
\end{proof}

Theorem \ref{t2.4} follows from Lemmas \ref{l5.3}, \ref{l5.4} and \ref{l5.5} immediately.

\section{Proof of Theorem \ref{t1.6}}

For a graph $G=(V,E)$, any $\beta>0$ and $r\in \mathbb{Z}^{+}$, define
  \bess
    D(\beta; r)=\sup_{x \in V}\frac{\mathcal{V}\big( B_x^r\cap\{y:\, \phi(y)\ge\beta\}\big)}{\mathcal{V}\big( B_x^r\big)},\;\;\;
   \bar{D}(\beta)= \limsup_{\mathbb{Z}^{+}\ni r\to +\infty}D(\beta;r).
  \eess

\begin{lemma}\label{l6.1} Let $G=(V, E)$ be an infinite locally finite graph. Suppose $D_\mu<+\infty$ and $\mu(x) \equiv 1$ on $V$. Assume that the heat kernel $P=P(t,\cdot,\cdot;G)$ on $G$ satisfies
\begin{equation}\label{6.1}
 P(t,x,y)=P(t,z,w)
 \end{equation}
for any $x,y,z,w\in V$ satisfying $d(x,y)=d(z,w)$ and any $t>0$. If there exist $m\in\mathbb{R}$ and a constant $A=A(G)>0$ such that
 \begin{equation}\label{6.2}
\mathcal{V}(x,r)=Ar^m+O\left(r^{m-1}\right),\;\;\forall\; x\in V,\; r\in\mathbb{R}\text {, }
 \end{equation}
and there exists $\beta>0$ so that $\bar{D}(\beta)>0$, then
 \bes\label{6.3}
    T_1 \le \frac{1}{p-1}(\beta \bar{D}(\beta))^{1-p},
 \ees
 where $T_1$ is the life span of the solution of the problem \eqref{1.6} with $\lambda=1$.
\end{lemma}

\begin{proof} {\it Step 1}. As $\bar{D}(\beta)>0$, for any $\varepsilon>0$, there exist $z_i\in V$, $s_i\in\mathbb{Z}^{+}$, $i=1,2,\cdots$, such that $\lim_{i \to\infty}s_i=+\infty$ and
 \bes
 \frac{\mathcal{V}\bigl(B_{z_i}^{s_i}\cap\{y:\, \phi(y)\ge\beta\}\bigr)}{\mathcal{V}(B_{z_i}^{s_i})}
\ge\bar{D}(\beta)-\varepsilon.\label{6.4}
\ees
In the following, we let $r_i=[\sqrt{s_i}]$, where $[\sqrt{s_i}]$ is the greatest integer not more than $\sqrt{s_i}$.

We claim that for all $\delta \in (0,1)$, there exists $K\in\mathbb{N}$ so that for any $i>K$
 \bes\label{6.5}
&&\sup_{x\in B_{z_i}^{s_i-r_i}}\sum_{y\in B_{z_i}^{s_i}}\chi_{B_x^{r_i}(y)}P(T_{1},x,y)\phi(y)\nm\\
&\ge&\frac{1-\delta}{\mathcal{V}(B_{z_i}^{s_i-r_i})}
\Big\{\beta\bigl(\bar{D}(\beta)
-\varepsilon\bigr)\mathcal{V}\left(B_{z_i}^{s_i}\right)-\|\phi\|_{\ell^{\infty}(V)}\bigl[\mathcal{V}\left( B_{z_i}^{s_i}\setminus B_{z_i}^{s_i-2r_i}\right)\bigr]\Big\}.
 \ees
Here $\chi_{E}$ is the characteristic function of the set $E$.

Define a sequence of functions $F_i:
    B_{z_i}^{s_i} \to [0,+\infty)$ as
    \begin{equation*}
        F_i(x)=\sum_{y\in B_{z_i}^{s_i}}\chi_{B_x^{r_i}(y)}P(T_{1},x,y)\phi(y).
    \end{equation*}
Since $B_{z_i}^{s_i-r_i}\subset B_{z_i}^{s_i}$, we have
  \bes\label{6.6}
    \sum_{x\in B_{z_i}^{s_i-r_i}}F_i(x)&=&\sum_{x\in B_{z_i}^{s_i}}\chi_{B_{z_i}^{s_i-r_i}(x)}F_i(x)\nm \\
    &=&\sum_{x\in B_{z_i}^{s_i}}\chi_{B_{z_i}^{s_i-r_i}(x)}
    \sum_{y\in B_{z_i}^{s_i}}\chi_{B_x^{r_i}(y)}P(T_{1},x,y)\phi(y)\nm\\
    &=&\sum_{y\in B_{z_i}^{s_i}} \Bigl(\sum_{x\in B_{z_i}^{s_i}}\chi_{B_{z_i}^{s_i-r_i}(x)} \chi_{B_x^{r_i}(y)}P(T_{1},x,y)\,\Bigr)\phi(y)\nm \\
    &=&\sum_{y\in B_{z_i}^{s_i}}I_i(y)\phi(y),	
    \ees
where
    \begin{equation*}
        I_i (y)=\sum_{x\in B_{z_i}^{s_i}}\chi_{B_{z_i}^{s_i-r_i}(x)}
        \chi_{B_x^{r_i}(y)}P(T_{1},x,y).
    \end{equation*}
Fix a vertex $z \in V$. Take advantage \eqref{6.1} it yields
 \bes\label{6.7}
    \sum_{x\in V}\chi_{B_y^{r_i}(x)}P(T_{1},x,y)
    =\sum_{x\in B_y^{r_i}}P(T_{1},x,y)
    =\sum_{x\in V\cap B_z^{r_i}}P(T_{1},z,x).
    \ees
On the other hand, due to $D_{\mu}<\infty$, by Proposition \ref{p2.4}, we have $\sum_{x\in V}P(T_{1},z,x)=1$. This combines with \eqref{6.7} implies that for any fixed $\delta>0$ there exists $R=R(\delta)>0$ such that
 \bess
  \sum_{x\in V}\chi_{B_y^{r_i}(x)}P(T_{1},x,y)
  \ge1-\delta,\;\;\forall\;y\in V,\;r_i>R.
    \eess
Since $r_i=[\sqrt{s_i}]\to\infty$ as $i\to\infty$, there exists $K>0$ such that $r_i>R$ for $i>K$. Thus,
 \begin{equation}\label{6.8}
\sum_{x\in V}\chi_{B_y^{r_i}(x)}P(T_{1},x,y)\ge 1-\delta~\text{for}~i>K.
    \end{equation}
In the following proof we shall always take $i>K$. For $y\in B_{z_i}^{s_i-2r_i}$, it is clear that $\chi_{B_y^{r_i}(x)}=\chi_{B_x^{r_i}(y)}$ and $B_y^{r_i}\subset B_{z_i}^{s_i-r_i}$. Then we have, by \eqref{6.8},
\begin{equation*}
\begin{split}
I_i(y)&=\sum_{x\in B_{z_i}^{s_i-r_i}}\chi_{B_x^{r_i}(y)}P(T_{1},x,y)\\
&=\sum_{x\in B_{z_i}^{s_i-r_i}}\chi_{B_y^{r_i}(x)}P(T_{1},x,y)\\
&=\sum_{x\in V}\chi_{B_y^{r_i}(x)}P(T_{1},x,y)\ge 1-\delta
\end{split}
    \end{equation*}
for $y\in B_{z_i}^{s_i-2r_i}$. Consequently,
 \bess
I_i(y)\ge\begin{cases}
1-\delta, \, &y \in B_{z_i}^{s_i-2r_i},\\
 0,& y \in B_{z_i}^{s_i} \setminus B_{z_i}^{s_i-2r_i}.
 \end{cases}\eess
This combines with \eqref{6.6} gives
\begin{equation*}
\begin{split}
\sum_{x\in B_{z_i}^{s_i-r_i}}F_i(x)
&=\sum_{y\in B_{z_i}^{s_i}}I_i(y)\phi(y)\\
&=\sum_{y\in B_{z_i}^{s_i-2r_i}}\!\!I_i(y)\phi(y)
+\sum_{y\in B_{z_i}^{s_i}\setminus B_{z_i}^{s_i-2r_i}}\!\!\!\!I_i(y)\phi(y)\\
&\ge\sum_{y\in B_{z_i}^{s_i-2r_i}}\!\!I_i(y)\phi(y)\,\ge
(1-\delta)\sum_{y\in B_{z_i}^{s_i-2r_i}}\phi(y).
        \end{split}
\end{equation*}
Let $x_i \in B_{z_i}^{s_i-r_i}$ so that $F_i(x_i)=\max\limits_{B_{z_i}^{s_i-r_i}}F_i(\cdot)$. It then  follows that
\begin{equation}\label{6.9}
 F_i(x_i)\ge\frac{1-\delta}{\mathcal{V}(B_{z_i}^{s_i-r_i})}
 \sum\limits_{y\in B_{z_i}^{s_i-2r_i}}\phi(y).
    \end{equation}
Make use of \eqref{6.4},
  \bess
\sum_{y\in B_{z_i}^{s_i}}\phi(y)\ge\beta\!\!\sum_{y\in B_{z_i}^{s_i}\cap\{y:\,\phi(y)\ge\beta\}}\!\!\!\!1
\ge\beta\mathcal{V}\bigl(B_{z_i}^{s_i}\cap\{x:\,\phi(x)\ge\beta\}\bigr)
\ge\beta\bigl(\bar{D}(\beta)-
\varepsilon\bigr)\mathcal{V}(B_{z_i}^{s_i}),
\eess
and hence
 \bess
\sum_{y\in B_{z_i}^{s_i-2r_i}}\phi(y)&=&\sum_{y\in B_{z_i}^{s_i}}\phi(y)-
\sum_{y\in B_{z_i}^{s_i}\setminus B_{z_i}^{s_i-2r_i}}\phi(y)\nm\\
&\ge&\beta\bigl(\bar{D}(\beta)-\varepsilon\bigr)\mathcal{V}
\bigl(B_{z_i}^{s_i}\bigl)-\|\phi\|_{\ell^{\infty}(V)}\mathcal{V}\bigl(B_{z_i}^{s_i}\setminus B_{z_i}^{s_i-2r_i}\bigl).
 \eess
Make use of this and \eqref{6.9}, it follows that
 \bess
F_i(x_i)\ge\frac{1-\delta}{\mathcal{V}(B_{z_i}^{s_i-r_i})}
\bigl[\beta\bigl(\bar{D}(\beta)-\varepsilon\bigr)
\mathcal{V}\left(B_{z_i}^{s_i}\right)
-\|\phi\|_{\ell^{\infty}(V)}\mathcal{V}\bigl(B_{z_i}^{s_i}\setminus B_{z_i}^{s_i-2r_i}\bigl)\bigl].
\eess
The inequality \eqref{6.5} is obtained.

{\it Step 2}. It follows from \eqref{6.5} that, for any $i>K$,
  \bes\label{6.10}
\sup_{z\in V}\sum_{y\in V}P(T_{1},z,y)\phi(y)&\ge&\sup_{x\in B_{z_i}^{s_i-r_i}}\sum_{y \in B_{z_i}^{s_i}}\chi_{B_x^{r_i}(y)}P(T_{1},x,y)\phi(y),\nonumber\\
&\ge&\frac{1-\delta}{\mathcal{V}(B_{z_i}^{s_i-r_i})}\big[\beta\bigl(\bar{D}(\beta)
-\varepsilon\bigr)\mathcal{V}(B_{z_i}^{s_i})-
\|\phi\|_{\ell^{\infty}(V)}\mathcal{V}\bigl(B_{z_i}^{s_i}\setminus B_{z_i}^{s_i-2r_i}\bigl)\big].\qquad
 \ees
Making use of \eqref{6.2} we have
 \bes
\mathcal{V}(B_{z_i}^{s_i}\backslash B_{z_i}^{s_i-2r_i})& =&\mathcal{V}(B_{z_i}^{s_i})-\mathcal{V}(B_{z_i}^{s_i-2r_i})\nm\\
&=&A s_i^m+O(s_i^{m-1})-\big(A(s_i-2 r_i)^m+O(s_i-2r_i)^{m-1}\big)\nm\\
&=:&M_i.\label{6.11}
 \ees
It follows from \eqref{6.2}, \eqref{6.10} and \eqref{6.11} that
 \bess
\sup_{z\in V}\sum_{y\in V}P(T_{1},z,y)\phi(y)\ge(1-\delta) \frac{\left[A s_i^m+O\left(s_i^{m-1}\right)\right]\beta(\bar{D}(\beta)-\varepsilon)
-M_i\|\phi\|_{l^{\infty}(V)}}{A\left(s_i-r_i\right)^m+O\left(s_i-r_i\right)}.
 \eess
Take $i\to\infty$ to get $\sup_{z\in V}\sum_{y\in V}P(T_{1},z,y)\phi(y)
\ge(1-\delta)\beta\bigl(\bar{D}(\beta)-\varepsilon\bigr)$. By the arbitrariness of $\varepsilon$ and $\delta$,
\begin{equation}\label{6.12}
\sup_{z\in V}\sum_{y\in V}P(T_{1},z,y)\phi(y)
    \ge \beta\bar{D}(\beta).
\end{equation}

By similar arguments as in the proof of Lemma \ref{l4.6}, it can be derived that
 \begin{equation}\label{6.13}
T_1\le \frac{1}{p-1}\Bigg(\sup_{z\in V}\sum_{y\in V}P(t,z,y)\phi(y)\Bigg)^{1-p}.
\end{equation}
The desired result \eqref{6.3} can be deduced by \eqref{6.12} and  \eqref{6.13}.
\end{proof}

\begin{lemma}\label{c6.1}
    If there exists $\beta>0$ so that $\bar{D}(\beta)>0$, then
\bes\label{6.14}
    T_1\left(\mathbb{Z}^{N},\Delta_{\mathbb{Z}^{N}}\right)\le \frac{1}{p-1}(\beta \bar{D}(\beta))^{1-p},
\ees
    where $T_1\left(\mathbb{Z}^{N},\Delta_{\mathbb{Z}^{N}}\right)$ is the life span of the solution to the problem \eqref{1.10} with $\lambda=1$.
\end{lemma}
\begin{proof}
It is well known that for any $r\in\mathbb{Z}^{+}$,
  $$\mathcal{V}(x,r)=\sum_{x\in \mathbb{Z}^d,\; d(x,0)\leq r}1=A(d)r^d+O\left(r^{d-1}\right).$$
According to \cite[Theorem 12]{CY2}, the heat kernel $P$ for $\left(\mathbb{Z}^d, \Delta_{\mathbb{Z}^d}\right)$ satisfies \eqref{6.1}. Thus, by Lemma \ref{l6.1},  \eqref{6.14} holds.
\end{proof}

\begin{proof}[Proof of Theorem \ref{t1.6}]
Without loss of generality, we suppose that
 \begin{equation*}
A=\liminf_{\mathbb{Z}\ni x\to +\infty}\phi(x)\ge\liminf_{\mathbb{Z}\ni x\to -\infty}\phi(x).
 \end{equation*}
Then for any $\varepsilon>0$ we can find $R\in\mathbb{Z}^{+}$ so that
$\phi(x)\ge A-\varepsilon$ for $x\ge R$.
Hence, it is easy to check that $\bar{D}(A-\varepsilon)=1$. According to Lemma \ref{c6.1},  $T_1\left(\mathbb{Z},\Delta_{\mathbb{Z}}\right)\le\frac{1}{p-1}(A-\varepsilon)^{1-p}$.
Taking $\varepsilon\to 0$ we get \eqref{1.11}.
\end{proof}


\vspace{50pt}

\end{document}